\documentclass[11pt,a4paper,fleqn]{article}

\usepackage{graphicx}
\usepackage{ifthen}
\usepackage{makeidx}
\usepackage{amsfonts}   % for mathbb{}
\usepackage{theorem}
\usepackage{psfrag}
\usepackage[square,numbers]{natbib}
\usepackage{amssymb,amsmath,euscript,mathrsfs}
\usepackage[Algorithm,ruled]{algorithm}
\usepackage[misc]{ifsym}

\usepackage{url}
\usepackage{hyperref}

\hypersetup{
	colorlinks=true, % This is crucial: it makes the links colored instead of boxed
	linkcolor=blue,  % Color of internal links (e.g., table of contents, citations)
	urlcolor=magenta,   % Color of external URLs (like those from the 'url' field in .bib)
	citecolor=blue,  % Color of citation links (this is what you want for references)
}

\usepackage{caption}  % Recommended for better caption control
\usepackage{subcaption} % For creating subfigures with their own captions

\usepackage{xcolor}
\definecolor{Zgris}{rgb}{0.87,0.85,0.85}

\usepackage{array}
\usepackage{epstopdf}
\usepackage{enumerate}

\newcommand{\RR}{{\textnormal{\sf I}}\!{\textnormal{\sf R}}}
\newcommand{\Rn}[1]{\ensuremath{\RR^{#1}}}

\newcommand{\vect}[1]{\boldsymbol{\mathrm{#1}}}
\newcommand{\matr}[1]{\mathsf{#1}}
\newcommand{\tran}[1]{#1^{\textnormal{\sf{T}}}}

\newcommand{\ru}{v}

\newcommand{\enk}{n}
\newcommand{\en}{\vect{\enk}}
\newcommand{\err}{\vect{\epsilon}}

\newtheorem{Lemma}{Lemma}
\newtheorem{Theorem}{Theorem}
\newtheorem{Definition}{Definition}
\newtheorem{Corollary}{Corollary}

\newtheorem{Example}{Example}

\newtheorem{Assumption}{Assumption}
\newtheorem{Remark}{Remark}

\newcommand{\pk}{u}
\newcommand{\p}{\vect{\pk}}

\newcommand{\np}{n_\pk}

\newenvironment{proof}
	{\begin{sloppypar}{\bf Proof.}}
	{\hspace*{\fill}$\Box$
	\end{sloppypar}}

\begin{document}

\title{Interpolation-Based Gradient-Error Bounds for Use in Derivative-Free Optimization of Noisy Functions}

\author{\bf Alejandro G. Marchetti$^{(1,2)}$, Dominique Bonvin$^{(3)}$\\ \\
   $^{(1)}$CIFASIS (CONICET - Universidad Nacional de Rosario),\\ S2000EZP Rosario, Argentina. \\
   \Letter$\ $ {\texttt marchetti\makeatletter @cifasis-conicet.gov.ar} \\ \\
   $^{(2)}$Facultad de Ciencias Exactas, Ingeniería y Agrimensura (FCEIA),\\ Universidad Nacional de Rosario (UNR), Rosario, S200BTP, Argentina. \\ \\
   $^{(3)}$Laboratoire d'Automatique, Ecole Polytechnique F\'ed\'erale de Lausanne,\\ CH-1015 Lausanne, Switzerland.\\ 
   {\texttt dominique.bonvin\makeatletter @epfl.ch} \\ \\}

\date{}  % This sets the date to be empty
\maketitle

\begin{abstract}
	In this paper, we analyze the accuracy of gradient estimates obtained by linear interpolation when the underlying function is subject to bounded measurement noise. The total gradient error is decomposed into a deterministic component—arising from the interpolation (finite-difference) approximation—and a stochastic component due to noise. Various upper bounds for both error components are derived and compared through several illustrative examples.
	Our comparative study reveals that strict deterministic bounds, including those commonly used in derivative-free optimization (DFO), tend to be overly conservative. To address this, we propose approximate gradient error bounds that aim to upper bound the gradient error norm more realistically, without the excessive conservatism of classical bounds. Finally, drawing inspiration from dual real-time optimization strategies, we present a DFO scheme based on sequential programming, where the approximate gradient error bounds are enforced as constraints within the optimization problem.\\


\end{abstract}

%%%%%%%%%%%%%%%%%%%%%%%%%%%%%%%%%%%%%%%%%%%%%%%%%%%%%%%%%%%%%%%%%%%%%%%%%%%%%
\section{Introduction}
\label{sec:Intro}

Derivative-free optimization (DFO) techniques are employed in scenarios where the objective function (and possibly the constraints) of the optimization problem lack explicit derivatives, or where obtaining derivatives is computationally or economically expensive, or inaccurate due to noise in function evaluations. Among the most popular DFO techniques are direct search methods and interpolation-based trust-region methods (see \cite{Conn:2009,Rios:2013} and references therein). 

Simplex gradients are computed by linear interpolation using function values at $n_u+1$ well-poised sample points in $\mathbb{R}^{n_u}$. These gradients are commonly used in DFO algorithms, including direct search methods such as the Nelder-Mead algorithm (see Chapter~8 in \cite{Conn:2009}) and line-search strategies based on simplex derivatives (see Chapter~9 in \cite{Conn:2009}). 

The {\it generalized simplex gradient} extends the definition of the simplex gradient to the {\it underdetermined} and {\it overdetermined} cases using fewer or more than $n_u+1$ points, respectively \cite{Custodio:2008,Regis:2015}. Variants such as the {\it centered simplex gradient} \cite{Kelley:1999} and the {\it generalized centered simplex gradient} \cite{Hare:2020} have also been proposed to improve approximation accuracy.
 
This paper focuses on the {\it determined} simplex gradient, i.e., the gradient of the unique linear function that interpolates a well-poised set of $n_u+1$ sample points. In this case, the simplex gradient is invariant to the ordering of the sample points and the linear model does not depend on the choice of the reference point from among the sample points. 

Our interest in the simplex gradient is motivated by its role in dual modifier-adaptation (MA) algorithms for real-time optimization (RTO) \cite{Marchetti:2010,Marchetti:2013a}.
To enhance industrial process operations, RTO repeatedly refines a first-principles model using real-time plant data. This updated model is then re-optimized to determine new input variables for the plant. Each RTO iteration involves intensive computations, such as solving nonlinear parameter estimation and economic optimization problems. While the underlying models often contain hundreds or thousands of variables, the number of independent manipulated variables—those directly controlled for optimization—is typically small, ranging from 2 to 10 across many relevant applications. Consequently, the scalability of RTO algorithms concerning these manipulated variables is generally not a major concern. Instead, the main focus is on maximizing economic profit (or minimizing cost), requiring as few experimental runs as possible, and guaranteeing constraint satisfaction. Given the inherent inaccuracies of most models, modifier-adaptation (MA) schemes refine the model by introducing first-order correction terms to its predicted objective and constraint functions \cite{Marchetti:09,Marchetti:2016}. This crucial step corrects the model's gradient predictions, allowing for convergence to the true plant optimum. Specifically, dual MA schemes calculate the simplex gradient based on data collected from the current and previous operating points \cite{Marchetti:2010,Marchetti:2013a}. They ensure the accuracy of this gradient by incorporating a {\it duality constraint} into the optimization problem, which specifically accounts for the gradient estimated from the subsequent simplex.

DFO schemes for noisy functions and MA schemes for RTO share significant similarities. Both are tailored to deal with expensive and noisy function evaluations, necessitating careful attention to the accuracy of gradient estimates. 
The total gradient error can be decomposed into a truncation error, arising from finite-difference approximation, and a measurement noise error \cite{Brekelmans:05}. Consequently, an upper bound on the gradient-error norm can be established by independently bounding these two components.

For the truncation error in the absence of noise, Conn et al. \cite{Conn:2009} (see Theorem 2.11 in \cite{Conn:2009}) provide an error bound for the simplex gradient, assuming Lipschitz continuity of the function's gradient. 
This bound, which we shall denote the {\it delta bound}, is currently the established error bound for the simplex gradient in the literature of derivative-free optimization \cite{Regis:2015,Berahas:2022,Jarry:2023}. 

In the context of dual RTO, a measure of gradient accuracy was proposed in \cite{Marchetti:2010} to control the simplex gradient error due to truncation. This measure, which we shall denote the {\it radial bound}, was used in \cite{Marchetti:2010} to develop a dual MA algorithm. The radial bound pays attention to gradient accuracy at a reference point that belongs to the set of interpolation points. An alternative approximate bound on the gradient error due to truncation that pays attention to gradient accuracy at at least one point that belongs to the convex hull determined by the interpolation points was proposed in \cite{Marchetti:2013a}. This bound, which we shall denote the {\it simplex bound}, was applied as a duality constraint in a dual MA scheme.

The impact of measurement noise on gradient estimation was analyzed in \cite{Brekelmans:05} for various gradient estimation schemes. Assuming the additive noise terms are independent and identically distributed (iid) with zero mean and a given variance, the expected value and variance of the gradient error norm is evaluated in \cite{Brekelmans:05} for forward finite differencing, central finite differencing, replicated central finite differencing, and the Plackett-Burman scheme. Additionally, \cite{More:2012} derives lower and upper bounds on the expected squared directional derivative error under iid random noise assumptions over an interval.
Turning to scenarios with bounded noise, Berahas et al. \cite{Berahas:2022} recently proposed an upper bound on the norm of the gradient error attributed to measurement noise, which we designate the {\it conditioning bound}. It is important to note, however, that the least upper bound for this specific error was previously established in \cite{Marchetti:2010}, which we refer to as the {\it l-min bound}.

In this paper, we derive and compare several gradient error bounds for the simplex gradient. We demonstrate that the radial bound provides an upper bound for the projection of the gradient error along the directions defined by the simplex vertices relative to a reference vertex. When these directions are orthogonal, the radial bound offers a strict upper bound on the norm of the truncation-induced gradient error. Furthermore, we introduce a novel strict upper bound on the truncation gradient error, termed the {\it square column bound}, and show it to be less conservative than the delta bound. Notably, the delta, radial, and square column bounds coincide in the specific case of forward finite differences.

Through comparative examples, we illustrate that strict bounds like the delta and square column bounds can often become overly conservative. To address this, we introduce the concept of {\it approximate gradient error bounds}, a category into which the radial and simplex bounds fall. We also extend the definition of the radial bound to evaluate gradient accuracy at non-vertex points within the simplex's interior. This {\it extended radial bound} allows for a more rigorous formalization of the simplex bound's definition.

The paper also revisits the l-min bound for bounding the gradient error  due to measurement noise and presents a refined proof. Our analysis shows that the l-min bound is less conservative than the conditioning bound, except in the special case of forward finite differences, where both bounds coincide. 

Finally, we introduce a sequential quadratic DFO scheme tailored for optimizing noisy functions in unconstrained settings. Our approach integrates the approximate gradient-error bounds as duality constraints within the optimization problem formulation. Through the implementation of this algorithm on two numerical examples, we demonstrate the effectiveness and versatility of employing duality constraints in the context of DFO for noisy functions.

The subsequent sections of the paper are structured as follows: Section~\ref{sec:preliminaries} provides preliminary introductory material. Section~\ref{sec:errorbounds} presents and derives the various gradient-error bounds. The concept of approximate gradient error bounds is introduced in Section~\ref{sec:approx}. The radial and simplex bounds are put into this category. In Section~\ref{sec:duality}, a comparative analysis of key gradient-error bounds is conducted in the two-dimensional case. Section~\ref{sec:DFOduality} introduces the proposed DFO algorithm with duality constraints. Finally, Section~\ref{sec:Conclusions} provides concluding remarks.

%%%%%%%%%%%%%%%%%%%%%%%%%%%%%%%%%%%%%%%%%%%%%%%%%%%%%%%%%%%%%%%%%%%%%%%%%%%%%
\section{Preliminaries}
\label{sec:preliminaries}

\subsection{Notation}

A column vector is denoted by a lower-case letter such as $\p$ and its elements are $u_i$, $i=1,\dots,n_u$. If the vector is denoted $\p_k$, then its elements are $u_{k,i}$, $i=1,\dots,n_u$. 
The $j$th unit vector, denoted as $\vect{e}_j$, is the vector with a length of 1 that points along the $j$th coordinate axis.
A matrix is denoted by an upper-case letter such as $\vect{U}$ and its elements are $U_{i,j}$. The gradient of a function $f(\p)$, denoted $\nabla f(\p)$, is a column vector. The identity matrix in $\mathbb{R}^{n_u\times n_u}$ is denoted $\vect{I}_{n_u}$. A diagonal matrix $\vect{D}\in\mathbb{R}^{n_u\times n_u}$ can be written as $\vect{D}=\text{diag}\{D_{1,1},D_{2,2},\dots,D_{n_u,n_u}\}$. The norm $\|\cdot\|$ denotes the Euclidean norm of a vector, or the spectral norm of a matrix (natural norm induced by the Euclidean norm).

\subsection{Assumptions and Definitions}

Gradient values depend on the order of magnitude of the input variables $\p$. It is assumed throughout that all the input variables $\p$ are of the same order of magnitude, which can be achieved via scaling. For example, if the input variable $\pk_i$ remains within the interval $[\pk_i^\text{min},\pk_i^\text{max}]$, it can be scaled as $\pk_i^{\text{scaled}}=(\pk_i-\pk_i^\text{min})/(\pk_i^\text{max}-\pk_i^\text{min})$. For notational simplicity, the superscript indicating a scaled variable will be omitted in the sequel.
\begin{Assumption}[Scaled Input Variables] All input variables in $\p$ are of the same order of magnitude.
\end{Assumption}

The analysis is carried out for a noisy function, for which only noisy evaluations are available,
\begin{align}
	& \tilde{f}(\p) = f(\p) + \ru(\p),  \label{eq:noisyfun}
\end{align}
where $\ru$ represents the measurement noise, which accounts for possible inaccuracies in the observed values of $f(\p)$.
We assume that the noise $\ru(\p)$ is bounded for all $\p\in\mathbb{R}^{n_u}$.

\begin{Assumption}[Bounded Noise] There exists a constant $\delta\geq 0$ such that $|\ru(\p)| \leq \delta$, for all $\p\in\mathbb{R}^{n_u}$.
	\label{assum:boundednoise}
\end{Assumption}

Furthermore, we assume that $\delta$ is known. If the noise is stochastic, then, based on a statistical description of $\ru$, $\delta$ could be selected by considering a desired confidence interval.\\

We assume that the function $f$ is twice continuously differentiable and that the gradient of $f$ is Lipschitz continuous on a convex set $\mathcal{Q}\subset \mathbb{R}^{n_u}$.

\begin{Definition}[Lipschitz Continuous Gradient]
	Let $\mathcal{Q}$ be a subset of $\mathbb{R}^{n_u}$ and let $f(\p)$ be a function defined on $\mathcal{Q}$. The gradient $\nabla f(\p)$ is Lipschitz continuous on $\mathcal{Q}$ with constant $L\geq 0$ if
	\begin{align}
		& \bigl\|\nabla f(\p_a) - \nabla f(\p_b)\bigr\| \leq L\|\p_a - \p_b\|,\ \ \text{\rm for all}\ \p_a,\p_b \in\mathcal{Q}. \label{eq:lip_Hf}
	\end{align}
\end{Definition}

\begin{Assumption}[Differentiability] The function $f(\p)$ is twice continuously differentiable on $\mathcal{Q}\subset \mathbb{R}^{n_u}$. \label{assum:differ}
\end{Assumption} 

\begin{Assumption}[Lipschitz Continuity of the Gradient of $f$] The gradient $\nabla f(\p)$ is Lipschitz continuous on a convex set $\mathcal{Q}\subset \mathbb{R}^{n_u}$.
	\label{assum:Lipschitz}
\end{Assumption}

We assume that a sample set $\mathcal{U} = \{\p_0,\ \p_1,\ \dots,\ \p_{n_u}\}$ in $\mathbb{R}^{n_u}$, comprising $n_u+1$ sample points in a convex set $\mathcal{Q}\subset \mathbb{R}^{n_u}$, is available for linear interpolation.

\begin{Definition}[Poisedness for Linear Intepolation]
	The set $\mathcal{U} = \{\p_0,\ \p_1,\ \dots,\ \p_{n_u}\}$ is poised for linear interpolation in $\mathbb{R}^{n_u}$ if the matrix
	\begin{align} 
		&\vect{U} := [\begin{array}{cccc}
			\p_1-\p_0, & \p_2-\p_0,& \dots, & \p_{n_u}-\p_0 
		\end{array}] \in\mathbb{R}^{{n_u}\times{n_u}}, \label{eq:U}
	\end{align}
	is nonsingular.
\end{Definition}

Poisedness of the sample set is required for linear interpolation.

\begin{Assumption}[Poisedness of the Sample Set] The set $\mathcal{U} = \{\p_0,\ \p_1,\ \dots,\ \p_{n_u}\}\subset\mathcal{Q}$ is poised for linear interpolation in $\mathbb{R}^{n_u}$.
	\label{assum:Poisedness}
\end{Assumption}

The assumption that $\mathcal{Q}$ is convex ensures that the convex hull determined by the sample points is included in $\mathcal{Q}$. In other words, $\{\p_0,\ \p_1,\ \dots,\ \p_{n_u}\}\subset\mathcal{Q}$ implies $\text{conv}\{\p_0,\ \p_1,\ \dots,\ \p_{n_u}\}\subset\mathcal{Q}$. Alternatively, we could have directly assumed that $\text{conv}\{\p_0,\ \p_1,\ \dots,\ \p_{n_u}\}\subset\mathcal{Q}$.

\subsection{Important Implications}

The result in the following lemma holds for functions with Lipschitz continuous gradient.
\begin{Lemma}\label{lemma:1}
	If the gradient of $f$ is Lipschitz continuous on a convex set $\mathcal{Q}\subset \mathbb{R}^{n_u}$, then, for all $\p_a,\p_b\in \mathcal{Q}$,
	\begin{align}
		& \bigl|f(\p_a) - f(\p_b) - \tran{\nabla f(\p_b)}(\p_a - \p_b) \bigr| \leq \frac{L}{2} \|\p_a - \p_b\|^2. \label{eq:lemma1}
	\end{align}
\end{Lemma}
\begin{proof}
	See Lemma~1.2.3 in \citep{Nesterov:2004}.
\end{proof}

\vspace{0.4cm}

Lipschitz continuity of $\nabla f$ implies that $f$ is a continuously differentiable function. If, in addition, the function $f$ is twice continuously differentiable, then a Lipschitz continuous gradient implies bounded eigenvalues on the Hessian, as stated in the following theorem.
\begin{Theorem} \label{th:spectral}
	Let Assumption~\ref{assum:differ} hold. Then, \eqref{eq:lip_Hf} is satisfied with $0 \leq L < \infty$ if and only if  $L$ is an upper bound on the spectral radius of the Hessian $\nabla^2 f(\p)$ for all $\p\in\mathcal{Q}$.
\end{Theorem}

\begin{proof}
	A proof can be found in \citep{Marchetti:2022}.
\end{proof}

\subsection{Linear Interpolation Model}

The linear interpolating polynomial has the form 
\begin{align}
	& m(\p) = c + \tran{\vect{g}}\p. \label{eq:funm}
\end{align}

Using a first-order approximation of $f$ in the neighborhood of $\p_0$, the value of $f$ at the points $\p_j$, $j=1,\dots,n_u$, is given by
\begin{align}
	& f(\p_j) = f(\p_0) + \tran{\nabla f(\p_0)}(\p_j - \p_0) + \mathcal{O}(\|\p_j-\p_0\|^2), \qquad j=1,\dots,n_u. \nonumber
\end{align}

An approximation $\vect{g} = \nabla m$ of the gradient $\nabla f(\p_0)$ satisfies
\begin{align}
	& \tilde{f}(\p_j) = \tilde{f}(\p_0) + \tran{\vect{g}}(\p_j - \p_0), \qquad j=1,\dots,n_u, \label{eq:interp}
\end{align}
which can be written as
\begin{align}
	& \tran{\vect{y}} = \tran{\vect{g}}\vect{U}, \nonumber
\end{align}
where $\vect{U}$ is defined in \eqref{eq:U}, and $\vect{y}$ is
\begin{align}
	& \vect{y} = \tran{\left(\tilde{f}(\p_1) - \tilde{f}(\p_0),\ \tilde{f}(\p_2) - \tilde{f}(\p_0),\ \dots,\ \tilde{f}(\p_{n_u}) - \tilde{f}(\p_0) \right)}.
\end{align}

Assuming that the sample set $\mathcal{U}$ is poised for linear interpolation, the gradient approximation $\vect{g}$ can then be computed as
\begin{align}
	& \vect{g} = \tran{[\vect{U}^{-1}]}\vect{y}, \label{simplexgrad}
\end{align}
which is known as the {\it simplex gradient}.

The linear interpolating polynomial is then given by
\begin{align}
	m(\p) = \tilde{f}(\p_0) + \tran{\vect{g}}(\p - \p_0). \label{eq:m_u0}
\end{align}

From \eqref{eq:funm}, \eqref{eq:interp}, and \eqref{eq:m_u0}, we have
\begin{align}
	& c = \tilde{f}(\p_0) - \tran{\vect{g}}\p_0 = \tilde{f}(\p_j) - \tran{\vect{g}}\p_j, \quad j=1,\dots,n_u. \nonumber
\end{align}

Hence, more generally, the linear interpolating polynomial can be written as
\begin{align}
	m(\p) = \tilde{f}(\p_j) + \tran{\vect{g}}(\p - \p_j), \qquad \text{for any}\ j\in\{0,1,\dots,n_u\}, \label{eq:m_general}
\end{align}
which shows that $m(\p)$ can be written using as a reference any of the points in the sample set $\mathcal{U}$.

%%%%%%%%%%%%%%%%%%%%%%%%%%%%%%%%%%%%%%%%%%%%%%%%%%%%%%%%%%%%%%%%%%%%%%%%%%%%%
\section{Gradient-Error Bounds for Linear Interpolation}
\label{sec:errorbounds}

We turn our attention to the quality of the simplex gradient $\vect{g}$ as an approximation to $\nabla f(\p)$, for which we define the gradient error
\begin{align}
	& \err(\p) := \vect{g} - \nabla f(\p),
\end{align}
which can be partitioned as
\begin{align}
	& \err(\p) = \err_t(\p) + \err_n,
\end{align}
with
\begin{align}
	& \err_t(\p) = \tran{[\vect{U}^{-1}]}\tran{[\ f(\p_1)-f(\p_0), \ \dots\ , \ f(\p_{n_u})-f(\p_0)\ ]} - \nabla f(\p) \label{equ:sop_d} \\
	& \err_n = \tran{[\vect{U}^{-1}]}\tran{[\ \ru_1-\ru_0, \quad \dots\ , \quad \ru_{n_u}-\ru_0\mbox{ }]}, \label{equ:sop_s} 
\end{align}
where $\err_t$ and $\err_n$ represent errors due to truncation and measurement noise, respectively \cite{Brekelmans:05}. The truncation error is due to the approximation incurred by using finite differences and is related to the curvature of $f(\p)$. The measurement noise error is due to the presence of measurement noise in $\tilde{f}(\p)$. 
The truncation gradient error can be written as
\begin{align}
	& \err_t(\p) = \vect{g}_t - \nabla f(\p) = \tran{[\vect{U}^{-1}]}\vect{y}_t - \nabla f(\p), \nonumber\\
	& \vect{y}_t = \tran{\left[f(\p_1) - f(\p_0),\ f(\p_2) - f(\p_0),\ \dots,\ f(\p_{n_u}) - f(\p_0) \right]}, \nonumber
\end{align}
where $\vect{g}_t$ is the simplex gradient obtained in the abscence of measurement noise.\\

The gradient-error norm satisfies the inequality
\begin{align}
	& \| \err(\p) \|  \leq \| \err_t(\p)\| + \| \err_n \|. \label{eq:normineq}
\end{align}

In Subsections \ref{ssec:truncation} and \ref{ssec:noise}, we shall analyze different upper bounds for the gradient-error norms that are due to truncation and measurement noise, respectively.

%%%%%%%%%%%%%%%%%%%%%%%%%%%%%%%%%%%%%%%%%
\subsection{Gradient Error due to Truncation}
\label{ssec:truncation}

In this subsection, we present various upper bounds on $\|\err_t\|$. All the bounds rely on the assumption of Lipschitz continuity of the gradient $\nabla f(\p)$.

\subsubsection{The Delta Bound}
\label{ssec:origin}

We start by considering the bound on $\|\err_t(\p_0)\|$ that is presented in \cite{Conn:2009}. This gradient-error bound has been extensively used in the literature of derivative-free optimization.

\begin{Theorem} \label{th:delta}
	Let Assumptions~\ref{assum:differ}, \ref{assum:Lipschitz}, and \ref{assum:Poisedness} hold. Then
		\begin{align}
			& \|\vect{g}_t - \nabla f(\p_0)\| \leq T_d := \|\hat{\vect{U}}^{-1}\|\sqrt{n_u}\frac{L}{2}\Delta, \label{eq:deltabound}
		\end{align}
		\begin{align}
			& \text{with}\quad \Delta = \max_{\p_j\in\ \mathcal{U}} \|\p_j - \p_0\|,\qquad\quad \text{and}\quad \hat{\vect{U}} = \frac{1}{\Delta}\vect{U}. \label{eq:deltadef} %\label{eq:deltadef}
	\end{align} 
\end{Theorem}

\begin{proof}
	The proof given in \cite{Conn:2009} uses the integral form of the mean-value theorem. For completeness, we present the proof using Lemma~\ref{lemma:1}.
	\begin{align}
		& |f(\p_j) - f(\p_0) - \tran{\nabla f(\p_0)}(\p_j-\p_0)| \leq \frac{L}{2}\|\p_j-\p_0\|^2, \quad j=1,\dots,n_u. \nonumber
	\end{align}
	
	Noting that $f(\p_j) = f(\p_0) + \tran{\vect{g}_t}(\p_j - \p_0)$, we have
	\begin{align}
		& \left|\tran{(\vect{g}_t - \nabla f(\p_0))}(\p_j-\p_0)\right| \leq \frac{L}{2}\|\p_j-\p_0\|^2 \leq \frac{L}{2}\Delta^2, \quad j=1,\dots,n_u, \nonumber
	\end{align}
	which can be written as the element-wise vector inequality
	\begin{align}
		& \left|\tran{(\vect{g}_t - \nabla f(\p_0))}\vect{U}\right| \leq \frac{L}{2}\Delta^2[1,\ 1,\ \dots\ ,\ 1]. \nonumber
	\end{align}
	
	Taking the norm of the inequality gives
	\begin{align}
		& \bigl\|\tran{(\vect{g}_t - \nabla f(\p_0))}\vect{U}\bigr\| \leq \sqrt{n_u}\frac{L}{2}\Delta^2. \label{eq:res1}
	\end{align}
	
	On the other hand, we have
	\begin{align}
		& \|\vect{g}_t - \nabla f(\p_0)\| =  \bigl\|\tran{(\vect{g}_t - \nabla f(\p_0))}\vect{U}\vect{U}^{-1}\bigr\| \leq \|\vect{U}^{-1}\|\bigl\|\tran{(\vect{g}_t - \nabla f(\p_0))}\vect{U}\bigr\|. \label{eq:res2}
	\end{align}
	
	Combining \eqref{eq:res1} and \eqref{eq:res2} gives
		\begin{align}
			& \|\vect{g}_t - \nabla f(\p_0)\| \leq  \|\vect{U}^{-1}\|\sqrt{n_u}\frac{L}{2}\Delta^2, \label{eq:deltabound2}
		\end{align}
		and the result in \eqref{eq:deltabound} follow from noting that $\hat{\vect{U}}^{-1} = \Delta\vect{U}^{-1}$. 
\end{proof}

\vspace{0.4cm}

We will refer to the truncation gradient-error bound \eqref{eq:deltabound} as the {\it delta bound}, and we will denote it by $T_d$.
The following remark is in order:
	\begin{Remark}
		If all the points in the sample set $\mathcal{U}$ are multiplied by a scaling factor $h>0$, that is, $\bar\p_j = h\p_j$, for $j=0,1,\dots,n_u$, then the scaled matrix $\vect{U}$ is $\bar{\vect{U}} = h\vect{U}$, and the scaled value of $\Delta$ is $\bar{\Delta} = h\Delta$. Hence, $\frac{\bar{\vect{U}}}{\bar{\Delta}} = \frac{\vect{U}}{\Delta} = \hat{\vect{U}}$ for all $h>0$, and the matrix $\hat{\vect{U}}$ does not depend on the scaling factor $h$.
\end{Remark} 

Note that all the sample points are contained in the ball $B(\p_0,\Delta)$ of radius $\Delta$. A uniform bound that is valid for any point $\p$ in the ball $B(\p_0,\Delta)$ is provided in \cite{Conn:2009}, as given in the following corollary. 
\begin{Corollary} \label{cor:delta}
	Let Assumptions~\ref{assum:differ}, \ref{assum:Lipschitz}, and \ref{assum:Poisedness} hold. Then, for all $\p \in B(\p_0,\Delta)$, we have
	\begin{align}
		& \|\vect{g}_t - \nabla f(\p)\| \leq \|\hat{\vect{U}}^{-1}\|\sqrt{n_u}\frac{L}{2}\Delta + L\Delta. \label{eq:deltabound_u}
	\end{align}
\end{Corollary}

\begin{proof}
	From the Lipschitz continuity of the gradient of $f$, we have\\
	
	$\|\vect{g}_t - \nabla f(\p)\| \leq \|\vect{g}_t - \nabla f(\p_0)\| + \|\nabla f(\p_0) - \nabla f(\p)\| \leq \|\hat{\vect{U}}^{-1}\|\sqrt{n_u}\frac{L}{2}\Delta + L\Delta$.
\end{proof}

\vspace{0.4cm}

The bound given in Corollary \ref{cor:delta} is valid for all $\p\in B(\p_0,\Delta)$, but it is clearly conservative. Indeed, when evaluated at $\p_0$, the bound is greater than that given by Theorem~\ref{th:delta}. Note that, if uniformity is not required, the following pointwise bound may be considered: 
\begin{align}
	& \|\vect{g}_t - \nabla f(\p)\| \leq \|\hat{\vect{U}}^{-1}\|\sqrt{n_u}\frac{L}{2}\Delta + L\|\p-\p_0\|. \label{eq:deltabound_u2}
\end{align}

\begin{Remark}[Delta bound for FFD] \label{rmk:FFDdelta}
	Let the sample points in $\mathcal{U}$ have the arrangement used in forward-finite differences (FFD), that is, $\p_j = \p_0 + h\vect{e}_j$, for $j=1,\dots,n_u$, where $h>0$ is the step size, and $\vect{e}_j$ is the $j^{th}$ unit vector. Then, the gradient-error bound in Theorem~\ref{th:delta} reduces to
	\begin{align}
		& \|\vect{g}_t - \nabla f(\p_0)\| \leq \frac{L \sqrt{n_u}\ h}{2}. \label{eq:deltaFFD}
	\end{align}
	Indeed, in the FFD arrangement, we have $\|\vect{U}^{-1}\| = 1/h$, and $\Delta = h$.
\end{Remark}

\subsubsection{The Radial Bound}
\label{ssec:radial}

Building on the preliminary work in \cite{Marchetti:2010}, we derive an upper bound on the norm of the gradient error due to truncation, applicable when the columns of the matrix $\vect{U}$ are orthogonal but not necessarily normalized. 
We refer to this bound as the {\it radial bound}, denoted by $T_r$. 

\begin{Theorem} \label{th:radial}
	Let Assumptions~\ref{assum:differ} and \ref{assum:Lipschitz} hold. If the columns in matrix $\vect{U}$ are orthogonal, then,
	\begin{align}
		& \|\vect{g}_t - \nabla f(\p_0)\| \leq T_r := L\ r(\mathcal{U}), \label{eq:radialbound}
	\end{align}
	where
	\begin{align}
		& r(\mathcal{U}) = \frac{1}{2} \Bigl\lVert  \bigl[ \ \tran{(\p_1-\p_0)}(\p_1-\p_0), \ \dots\ , \ \tran{(\p_{n_u}-\p_0)}(\p_{n_u}-\p_0) \ \bigr] \ \vect{U}^{-1} \Bigr\rVert \label{eq:radius}
	\end{align}
	is the radius of the unique sphere that contains all the sample points on its surface.
\end{Theorem}

\begin{proof}
	Applying Taylor's theorem in the neighborhood of $\p_0$ gives
	\begin{align}
		f(\p_j)-f(\p_0) =&\ \tran{\nabla f(\p_0)}(\p_j-\p_0) + \frac{1}{2}\tran{(\p_j-\p_0)}\nabla^2f(\bar{\p}_j)(\p_j-\p_0),  \label{equ:sop_d_a}
	\end{align}
	with $\bar{\p}_j = \p_0+\vartheta_j(\p_j-\p_0)$, for some $\vartheta_j \in [0,\ 1]$, for $j = 1,\dots,\np$. Substituting \eqref{equ:sop_d_a} into \eqref{equ:sop_d}, and upon evaluating at $\p = \p_0$, we obtain
	\begin{align}
		\tran{\err_t(\p_0)} &= \left( \tran{\nabla f(\p_0)}\ \vect{U} + \frac{1}{2} \left[ \tran{(\p_1-\p_0)}\nabla^2f(\bar{\p}_1)(\p_1-\p_0), \ \cdots \right. \right. \label{equ:sop_d_c} \\
		& \left. \left. \qquad \cdots \ \tran{(\p_{n_u}-\p_0)}\nabla^2f(\bar{\p}_{\np})(\p_{n_u}-\p_0) \right] \right) \ \vect{U}^{-1} - \tran{\nabla f(\p_0)} \nonumber \\
		& = \frac{1}{2} \left[ \tran{(\p_1-\p_0)}\nabla^2f(\bar{\p}_1)(\p_1-\p_0), \ \cdots\ ,\  \tran{(\p_{n_u}-\p_0)}\nabla^2f(\bar{\p}_{\np})(\p_{n_u}-\p_0) \right] \ \vect{U}^{-1}. \nonumber
	\end{align}
	Since the Hessian matrix $\nabla^2f(\bar{\p}_j)$ is symmetric, there exists an orthogonal matrix $\vect{P}_j$ such that $\tran{\vect{P}_j}\nabla^2f(\bar{\p}_j)\vect{P}_j=\vect{D}_j=\text{diag}(\lambda_{1,j},\dots,\lambda_{\np,j})$, where $\lambda_{i,j}\in\sigma\bigl(\nabla^2f(\bar{\p}_j)\bigr)$, and $\sigma(.)$ refers to the spectrum of a matrix. Letting $\vect{w}_j=\tran{\vect{P}_j}(\p_j-\p_0)$, for $j=1,\dots,n_u$, gives
	\begin{align} 
		& \bigl\lvert \tran{(\p_j-\p_0)}\nabla^2f(\bar{\p}_j)(\p_j-\p_0) \bigr\rvert = \bigl\lvert  \tran{\vect{w}_j}\tran{\vect{P}_j}\nabla^2f(\bar{\p}_j)\vect{P}_j\vect{w}_j \bigr\rvert = \bigl\lvert \tran{\vect{w}_j}\vect{D}_j\vect{w}_j \bigr\rvert \nonumber\\
		& = \Bigl\lvert \sum_{i=1}^{\np} \lambda_{i,j}w_{i,j}^2 \Bigr\rvert \leq L \sum_{i=1}^{\np}w_{i,j}^2  = L \tran{\vect{w}_j}\vect{w}_j = L \tran{(\p_j-\p_0)}(\p_j-\p_0), \quad j  = 1,\dots,\np, \label{eq:spect1} 
	\end{align}
	where the inequality follows from Theorem~\ref{th:spectral}.
	Let us introduce the matrix $\vect{R}:=\vect{U}^{-1}\tran{[\vect{U}^{-1}]}$ and the vectors
	\begin{align} 
		& \tran{\vect{p}_1} :=  \left[ \tran{\vect{w}_1}\vect{D}_1\vect{w}_1 \ \dots \ \tran{\vect{w}_{\np}}\vect{D}_{\np}\vect{w}_{\np} \right], \qquad\quad \tran{\vect{p}_2} :=  \left[ L\tran{\vect{w}_1}\vect{w}_1 \ \dots \ L\tran{\vect{w}_{\np}}\vect{w}_{\np} \right]. \nonumber 
	\end{align}
	
	Next, we show that $\vect{R}$ is a diagonal matrix with positive elements. 
	Given that $\vect{U}$ is a square matrix with orthogonal but not necessarily normalized columns, the columns of $\vect{U}$ satisfy
	\begin{align} 
		& \tran{\vect{U}}\vect{U} = \vect{D} = \text{diag}\{\|\p_1-\p_0\|^2,\|\p_2-\p_0\|^2,\dots,\|\p_{n_u}-\p_0\|^2\}. \nonumber
	\end{align}
	
	Since $\vect{U}$ has independent columns, Assumption~\ref{assum:Poisedness} is satisfied, and $\vect{U}$ and $\vect{D}$ are invertible. Noting that $\vect{D}^{-1}\tran{\vect{U}}\vect{U} = \vect{I}$, it follows that $\vect{U}^{-1}=\vect{D}^{-1}\tran{\vect{U}}$, and $\vect{R}=\vect{U}^{-1}\tran{[\vect{U}^{-1}]} = \vect{D}^{-1}$ is a diagonal matrix, whose elements are the reciprocals of the squared norms of the columns of $\vect{U}$.
	
	From \eqref{eq:spect1}, the absolute value of each element of $\vect{p}_1$ is lower than or equal to the corresponding element of $\vect{p}_2$, and since $\vect{R}$ is a diagonal matrix with positive elements, it follows that
	\begin{align}
		& \tran{\vect{p}_1}\vect{R}\ \vect{p}_1 \leq \tran{\vect{p}_2}\vect{R}\ \vect{p}_2. \label{eq:inequality}
	\end{align}
	
	Noting that
	\begin{align}
		& \tran{\vect{p}_1}\vect{R}\ \vect{p}_1 = 4 \|\err_t(\p_0)\|^2, \nonumber
	\end{align}
	and
	\begin{align}
		& \tran{\vect{p}_2}\vect{R}\ \vect{p}_2 = L^2 \bigl\lVert  \bigl[ \ \tran{(\p_1-\p_0)}(\p_1-\p_0), \ \dots\ , \ \tran{(\p_{n_u}-\p_0)}(\p_{n_u}-\p_0) \ \bigr] \ \vect{U}^{-1} \bigr\rVert^2, \nonumber
	\end{align}
	the inequality \eqref{eq:inequality} reduces to the gradient bound in \eqref{eq:radialbound}. It remains to be shown that $r(\mathcal{U})$ is the radius of the sphere that contains all the sample points.\\
	
	Let $\p_c$ be the center of a sphere of radius $r = r(\mathcal{U})$. We have
	\begin{align}
		& r^2 = \tran{(\p_0 - \p_c)}(\p_0 - \p_c)  \label{equ:sph_1}  \\
		& r^2 = \tran{(\p_j - \p_c)}(\p_j - \p_c), \qquad j = 1,\dots,\np.   \label{equ:sph_2} 
	\end{align}
	
	After expanding the right-hand sides of equations \eqref{equ:sph_1} and \eqref{equ:sph_2}, substracting \eqref{equ:sph_1} from  \eqref{equ:sph_2} and rearranging, we obtain
	\begin{align}
		& \tran{\p_j}\p_j - \tran{\p_0}\p_0 = 2\ \tran{\p_c}(\p_j - \p_0), \qquad j = 1,\dots,\np.  \label{equ:sph_3} 
	\end{align}

	Equations \eqref{equ:sph_3} can be written in matrix form as follows
	\begin{align}
		& \bigl[\tran{\p_1}\p_1 - \tran{\p_0}\p_0, \ \dots\ , \ \tran{\p_{n_u}}\p_{n_u} - \tran{\p_0}\p_0 \bigr] = 2 \tran{\p_c} \vect{U}. \label{equ:sph_4} 
	\end{align}
	
	From \eqref{equ:sph_4} and the  expansion
	\begin{align}
		\tran{(\p_j - \p_0)}(\p_j - \p_0) & =  \tran{\p_j}\p_j - 2\tran{\p_0}\p_j + \tran{\p_0}\p_0 \nonumber\\
		& = \tran{\p_j}\p_j - \tran{\p_0}\p_0 - 2\tran{\p_0}(\p_j-\p_0), \quad j = 1,\dots,\np,  \nonumber
	\end{align}
	we can write
	\begin{align}
		& \bigl[ \tran{(\p_1-\p_0)}(\p_1-\p_0), \ \dots\ , \ \tran{(\p_{n_u}-\p_0)}(\p_{n_u}-\p_0) \bigr] \vect{U}^{-1} = \nonumber \\
		& = \bigl[\tran{\p_1}\p_1 - \tran{\p_0}\p_0, \ \dots\ , \ \tran{\p_{n_u}}\p_{n_u} - \tran{\p_0}\p_0 \bigr]\vect{U}^{-1} - 2 \tran{\p_0} \bigl[ \p_1 - \p_0,\ \dots\ ,\ \p_{n_u}-\p_0 \bigr]\vect{U}^{-1} \nonumber\\
		& =  2(\tran{\p_c} - \tran{\p_0}) \label{equ:sph_5}
	\end{align}
	
	Hence,
	\begin{align}
		& \frac{1}{2} \Bigl\lVert  \bigl[ \ \tran{(\p_1-\p_0)}(\p_1-\p_0), \ \dots \ ,\ \tran{(\p_{n_u}-\p_0)}(\p_{n_u}-\p_0) \ \bigr] \vect{U}^{-1} \Bigr\rVert =  \bigl\lVert \p_c - \p_0 \bigr\rVert = r, \label{equ:sop_c}
	\end{align}
	which completes the proof.
\end{proof}

\begin{Remark}[Radial bound for FFD] \label{rmk:FFDradial}
	If, in Theorem~\ref{th:radial}, the points have the FFD arrangement $\p_j = \p_0 + h\vect{e}_j$, for $j = 1,\dots,n_u$, with the step size $h>0$, then we have $\tran{(\p_j-\p_0)}(\p_j-\p_0) = h^2$, $j=1,\dots,n_u$, and $\vect{U} = \text{\rm diag}(h,h,\dots,h)$. Hence,
	\begin{align} 
		&  r(\mathcal{U}) =  \frac{1}{2} \Bigl\lVert  \bigl[ \ h^2 \ h^2 \ \dots \ h^2 \ \bigr] \text{\rm diag}\Bigl(\frac{1}{h}, \frac{1}{h},\dots,\frac{1}{h}\Bigr) \Bigr\rVert = \frac{1}{2}\sqrt{\np}\ h, \nonumber
	\end{align}
	and the radial bound reduces to
	\begin{align}
		& \|\tran{\vect{g}_t} - \nabla f(\p_0)\| \leq \frac{L \sqrt{n_u}\ h}{2}. \nonumber
	\end{align}
\end{Remark}

Remarks \ref{rmk:FFDdelta} an \ref{rmk:FFDradial} show that the delta bound and the radial bound are identical in the case of the FFD arrangement.

The radial bound does not in general bound the gradient error norm if the columns of matrix $\vect{U}$ are not orthogonal. However, it does bound the projection of the gradient error in all the directions $\p_j-\p_0$ given by the columns of $\vect{U}$, as stated in the following corollary.

\begin{Corollary} \label{cor:radial}
	Let Assumptions~\ref{assum:differ}, \ref{assum:Lipschitz}, and \ref{assum:Poisedness} hold. Then, the gradient error projections, in the directions given by the columns of $\vect{U}$, are bounded by the radial bound, that is: 
	\begin{align}
		\Bigl|\tran{\err_t(\p_0)}\frac{\p_j-\p_0}{\|\p_j-\p_0\|} \Bigr| \leq T_r, \qquad j=1,\dots,n_u. \label{eq:corol_radial}
	\end{align}
\end{Corollary}

\begin{proof}
	$\ $ Taking the dot product of Equation~\eqref{equ:sop_d_c} with $(\p_j-\p_0)$, and noting that $\vect{U}^{-1}(\p_j-\p_0) = \vect{e}_j$, we get
	\begin{align}
		& \tran{\err_t(\p_0)}(\p_j-\p_0) = \frac{1}{2}\tran{(\p_j-\p_0)}\nabla^2f(\bar{\p}_j)(\p_j-\p_0),\qquad j=1,\dots,n_u. \label{eq:cor2a}
	\end{align}
	
	Combining \eqref{eq:cor2a} and \eqref{eq:spect1} gives
	\begin{align}
		& \bigl|\tran{\err_t(\p_0)}(\p_j-\p_0)\bigr| \leq \frac{L}{2}\|\p_j-\p_0\|^2 \qquad j=1,\dots,n_u. \label{eq:cor2b}
	\end{align}
	
	Hence, the projected gradient satisfies:
	\begin{align}
		\Bigl|\tran{\err_t(\p_0)}\frac{\p_j-\p_0}{\|\p_j-\p_0\|} \Bigr| &\leq \frac{L}{2}\|\p_j-\p_0\| = \frac{L}{2}\|(\p_j - \p_c) -  (\p_0 - \p_c)\| \nonumber\\
		&\leq \frac{L}{2}\bigl(\|\p_j - \p_c\| +  \|\p_0 - \p_c\|\bigr) = L r = T_r, \qquad j=1,\dots,n_u, \label{eq:cor2c}
	\end{align}
	and the proof is complete.
\end{proof}

\begin{Remark} \label{rmk:radialuniform}
The radial bound takes the same value, $T_r = L\cdot r(\mathcal{U})$, when evaluated at any of the $n_u+1$ points in the interpolation set. This is because the radius $r(\mathcal{U})$, which determines the bound, depends only on the geometry of the sample set $\mathcal{U}$ and not on which vertex is chosen as the base point for computing the simplex gradient.
\end{Remark}

\subsubsection{The Square Column Bound}
\label{ssec:origin}

We now introduce a novel upper bound on the gradient error due to truncation that uses the same assumptions as the delta bound and is also valid for any poised sample set.

\begin{Theorem} \label{th:colnorm}
	Let Assumptions~\ref{assum:differ}, \ref{assum:Lipschitz}, and \ref{assum:Poisedness} hold. Then
	\begin{align}
		& \|\vect{g}_t - \nabla f(\p_0)\| \leq T_c := \frac{L}{2} \bigl\lVert \ \tran{(\p_1-\p_0)}(\p_1-\p_0), \ \dots\ , \ \tran{(\p_{n_u}-\p_0)}(\p_{n_u}-\p_0) \ \bigr\rVert\|\vect{U}^{-1}\| \label{eq:colnormbound}
	\end{align}
\end{Theorem}

\begin{proof}
	$\ $Equation~\eqref{equ:sop_d_c} can be rewritten as
	\begin{align}
		& \tran{\err_t(\p_0)}(\p_j - \p_0) = \frac{1}{2}\tran{(\p_j-\p_0)}\nabla^2 f(\bar{\p}_j)(\p_j-\p_0),\qquad j=1,\dots,n_u, \nonumber
	\end{align}
	with $\bar{\p}_j = \p_0+\vartheta_j(\p_j-\p_0)$, for some $\vartheta_j \in [0,\ 1]$, for $j = 1,\dots,\np$. Taking the absolute value and using the inequalities \eqref{eq:spect1} gives
	\begin{align}
		& |\tran{\err_t(\p_0)}(\p_j - \p_0)| \leq 	\frac{L}{2}\tran{(\p_j-\p_0)}(\p_j-\p_0),\qquad j=1,\dots,n_u. \nonumber
	\end{align}
	
	Noting that $\tran{\err_t(\p_0)}(\p_j - \p_0)$ is the $j$th element of the row vector $\tran{\err_t(\p_0)}\vect{U}$ it follows that
	\begin{align}
		& \|\tran{\err_t(\p_0)}\vect{U}\| \leq \frac{L}{2} \bigl\lVert \ \tran{(\p_1-\p_0)}(\p_1-\p_0), \ \dots\ , \ \tran{(\p_{n_u}-\p_0)}(\p_{n_u}-\p_0) \ \bigr\rVert. \label{eq:normtbound}
	\end{align}
	
	Combining \eqref{eq:normtbound} and \eqref{eq:res2} gives \eqref{eq:colnormbound}.
\end{proof}

\vspace{0.3cm}

We will refer to the truncation gradient-error bound \eqref{eq:colnormbound} as the {\it square column bound}, since it includes the squared norms of the columns of $\vect{U}$, and we will denote it by $T_c$. 

\begin{Remark}[Square column bound for FFD] \label{rmk:FFDsquarecol}
	If, in Theorem~\ref{th:colnorm}, the points have the FFD arrangement $\p_j = \p_0 + h\vect{e}_j$, for $j = 1,\dots,n_u$, with the step size $h>0$, then we have $\tran{(\p_j-\p_0)}(\p_j-\p_0) = h^2$, $j=1,\dots,n_u$, and $\|\vect{U}^{-1}\| = 1/h$. Hence,
	\begin{align} 
		&  T_c =  \frac{L}{2} \bigl\lVert  \bigl[ \ h^2 \ h^2 \ \dots \ h^2 \ \bigr] \bigr\rVert \frac{1}{h} = \frac{L\sqrt{n_u}\ h}{2}. \nonumber
	\end{align}
\end{Remark}

Remark \ref{rmk:FFDsquarecol} shows that the square column bound is equal to both the delta bound and the radial bound for the FFD arrangement.
The following corollary proves that the square column bound is never more conservative than the delta bound.

\begin{Corollary}[Relation to the delta bound] \label{cor:equaldistant}
	Let Assumptions~\ref{assum:differ}, \ref{assum:Lipschitz}, and \ref{assum:Poisedness} hold. Then the error norm of the simplex gradient satisfies
	\begin{align}
		& \|\vect{g}_t - \nabla f(\p_0)\| \leq T_c \leq T_d, 
	\end{align}
	where $T_c = T_d$ if and only if the points $\p_j$, for $j = 1,\dots,n_u$, are all equally distant from $\p_0$.
\end{Corollary}

\begin{proof}
	From the definition of the square column bound in \eqref{eq:colnormbound}, we have:
	\begin{align} 
		T_c & =  \frac{L}{2} \Bigl\lVert \|\p_1-\p_0\|^2, \ \dots\ , \ \|\p_{n_u}-\p_0\|^2  \Bigr\rVert\|\vect{U}^{-1}\|. \nonumber
	\end{align}
	
	By definition of $\Delta$ in \eqref{eq:deltadef}, it follows that
	\begin{align}
		& \|\p_j-\p_0\|\leq\Delta\quad\Rightarrow \quad \|\p_j-\p_0\|^2\leq\Delta^2, \quad \text{\rm for all}\ j. \nonumber
	\end{align}

	Therefore,
	\begin{align}
		& T_c \leq \frac{L}{2} \bigl\lVert  \Delta^2 \ \Delta^2 \ \dots \ \Delta^2  \bigl\rVert \|\vect{U}^{-1}\| =
		\|\vect{U}^{-1}\|\sqrt{n_u}\frac{L}{2}\Delta^2 = \|\hat{\vect{U}}^{-1}\|\sqrt{n_u}\frac{L}{2}\Delta = T_d. \nonumber
	\end{align}
	
	Now, assume that all points $\p_j$ are equally distant from $\p_0$, i.e., $\|\p_j-\p_0\| = \Delta$ for all $j$. Then,
	\begin{align} 
		&  T_c =  \frac{L}{2} \bigl\lVert \Delta^2 \ \Delta^2 \ \dots \ \Delta^2 \bigl\rVert \|\vect{U}^{-1}\| = T_d, \nonumber
	\end{align}
	which proves the ``if" part. 
	
	Conversely, if the distances $\|\p_j-\p_0\|$ are not equal, then at least one is strictly less than $\Delta$, and hence
	\begin{align}
	& \Bigl\lVert \|\p_1-\p_0\|^2, \ \dots\ , \ \|\p_{n_u}-\p_0\|^2 \Bigr\rVert < \bigl\lVert  \Delta^2 \ \Delta^2 \ \dots \ \Delta^2 \bigl\rVert, \nonumber
	\end{align}
	implying $T_c < T_d$, which completes the proof.
\end{proof}

\subsubsection{Minimum Vertex Gradient-Error Bounds}

Many DFO methods are based on function evaluations over sequences of simplices \cite{Bortz:1998}. 
At each iteration, one point of the simplex is replaced to generate a new simplex. Consequently, we argue that convergence to a neighborhood of an optimal point occurs with the last simplex, rather than with the last iterate. Since the simplex gradient provides an estimate at any vertex of the simplex, it is sufficient to ensure gradient accuracy at at least one of these vertices.

With this in mind, we define the Minimum Vertex Gradient Error Bound as the smallest bound obtained over all the vertices of the simplex. Specifically, the {\it minimum vertex delta bound} is defined as
\begin{align}
	& T_{rv} = \min\{T_r(\p_0),\ T_r(\p_1),\ \dots,\ T_r(\p_{n_u})\}, \nonumber
\end{align}
and the {\it minimum vertex square column bound} as
\begin{align}
	& T_{cv} = \min\{T_c(\p_0),\ T_c(\p_1),\ \dots,\ T_c(\p_{n_u})\}. \nonumber
\end{align}

These bounds guarantee that the simplex gradient error is controlled at least at one vertex of the interpolation set.

\subsubsection{Examples}

\begin{Example} \label{ex:1}
	Consider the function
	\begin{align}
		& f(\p) = 2 u_1^2 - u_1u_2 + u_2^2 - 2u_1 + 1.4^{(2u_1 + u_2)}, \nonumber
	\end{align}
	with $\p = \tran{[u_1,\ u_2]}$. A gradient Lipschitz constant that is valid in $\mathcal{Q}=\{\p\in\mathbb{R}^2:\ 0\leq u_1\leq 1,\ 0\leq u_2 \leq 1\}$ is $L = 5.3$. 
	The simplex gradient is computed using the points $(\p_0,\ f(\p_0))$, $(\p_1,\ f(\p_1))$, and $(\p_2,\ f(\p_2))$, where $\p_0 = \tran{[0.5,\ 0]}$, $\p_1 = \tran{[0,\ 1]}$, and $\p_2 = \tran{[1,\ 0]}$. This yields the gradient estimate $\vect{g}_t = \tran{[1.12,\ 2.56]}$. The radius defined by these three points is $r = 0.7906$. Table~\ref{table:Ex1} reports the gradient error norm as well as the delta bound, the square column bound, and the radial bound evaluated at $\p_0$, $\p_1$, and $\p_2$. Both the delta and square column bounds provide distinct highly conservative upper bounds at each point, particularly at $\p_1$. However, the square column bound is consistently less conservative than the delta bound. In contrast, the radial bound provides the same upper bound value at all three points—i.e., a uniform bound—which is also significantly less conservative. Notably, this holds even though the matrix $\vect{U}$ is not orthogonal at any of the points.
\end{Example}

Example~\ref{ex:1} illustrates a common scenario in which the radial bound, while not guaranteed to upper bound the gradient error norm, often does so in practice, whereas the delta and square column bounds, which are valid upper bounds by construction, tend to yield significantly more conservative results.
Next, Example~\ref{ex:2} illustrates Corollary~\ref{cor:equaldistant}, and presents a case where the radial bound fails to upper bound  the gradient error.

\begin{table}[ht]
	\caption{Evaluation of truncation gradient error bounds in Example~\ref{ex:1}} % title of Table
	\centering % used for centering table
	\begin{tabular}{c c c c c} % centered columns (4 columns)
		\hline %inserts double horizontal lines
		& $\|\err_t(\p_i)\|$ & $T_d$ & $T_c$ & $T_r$ \\ [0.3ex] % inserts table
		%heading
		\hline % inserts single horizontal line
		$\p_0$ & 2.8443 & 10.72 & 7.73 & 4.19 \\ 
		$\p_1$ & 4.1788 & 26.7 & 22.26 & 4.19 \\
		$\p_2$ & 3.1386 & 21.89 & 15.6 & 4.19 \\  % [1ex] adds vertical space
		\hline %inserts single line
	\end{tabular}
	\label{table:Ex1} % is used to refer this table in the text
\end{table}

\begin{Example} \label{ex:2}
	Consider the function
	\begin{align}
		& f(\p) = u_1^2 + 6 u_2, \nonumber
	\end{align}
	with $\p = \tran{[u_1,\ u_2]}$. For this function, the gradient is Lipschitz continuous with Lipschitz constant $L=2$. 
	The simplex gradient is computed using the points $(\p_0,\ f(\p_0))$, $(\p_1,\ f(\p_1))$, and $(\p_2,\ f(\p_2))$, where $\p_0 = \tran{[0,\ \theta]}$, $\theta\in\mathbb{R}\setminus \{0.5\}$, $\p_1 = \tran{[-1,\ 0.5]}$, and $\p_2 = \tran{[1,\ 0.5]}$. With this configuration, the square column bound $T_c$ and the delta bound $T_d$ are equal at $\p_0$, since the points $\p_1$ and $\p_2$ are equidistant from $\p_0$ (see Corollary~\ref{cor:equaldistant}). Evaluating
	\begin{align}
		& \vect{U}^{-1} = \left[
		\begin{array}{rr}
			-0.5 & -\frac{1}{2\theta-1}  \\
			0.5  & -\frac{1}{2\theta-1} 
		\end{array}
		\right],\qquad \text{and} \quad \Delta^2 = \bigl(\theta-\frac{1}{2}\bigr)^2 +1, \nonumber
	\end{align}
	the radial bound as a function of $\theta$ becomes
	\begin{align}
		& T_r = \frac{2(\theta^2 - \theta + \frac{5}{4})}{|2\theta-1|}. \nonumber
	\end{align}

	Next, compute 
	\begin{align}
		& \tran{(\vect{U}^{-1})}\vect{U}^{-1} = \left[
		\begin{array}{lc}
			0.5 & 0  \\
			0   & \frac{2}{(2\theta-1)^2} 
		\end{array}
		\right], \nonumber
	\end{align}
	from which it follows that the spectral norm of $\vect{U}^{-1}$ depends on the value of $\theta$ as:
	\begin{equation} 
		\begin{array}{ll}
		\|\vect{U}^{-1}\| = \frac{\sqrt{2}}{|2\theta-1|}, & \text{if}\ \theta\in\mathcal{S},\\
		\|\vect{U}^{-1}\| = \frac{1}{\sqrt{2}}, & \text{if}\ \theta\notin\mathcal{S},
		\end{array} \nonumber
	\end{equation}
	where $\mathcal{S} = \bigl\{x\in [-0.5,\ 1.5]\setminus \{0.5\} \bigr\}$. Hence, the square column bound $T_c$ is given by:
	\begin{equation} 
		\begin{array}{ll}
			T_c = T_r, & \text{if}\ \theta\in\mathcal{S},\\
			T_c = \theta^2 - \theta + \frac{5}{4}, & \text{if}\ \theta\notin\mathcal{S}.
		\end{array} \nonumber
	\end{equation}
	
	In particular, setting $\theta = \frac{2}{3}$ results in a simplex gradient estimate $\vect{g}_t = \vect{0}$, which is a poor approximation of the true gradient. The radius defined by the three points is $r = 3.0833$.
	Table~\ref{table:Ex2} reports the gradient error norm, along with the delta, square column, and radial bounds, each evaluated at $\p_0$, $\p_1$, and $\p_2$.
	
	At $\p_0$, all three bounds yield the same value of 6.1667, which slightly overestimates the actual gradient error norm of 6. The equality $T_c = T_d$ holds because $\p_1$ and $\p_2$ are equidistant from $\p_0$, while $T_c = T_r$ follows from the fact that $\theta = \frac{2}{3} \in \mathcal{S}$.
	 
	At $\p_1$ and $\p_2$, the gradient error norm is 6.3246. The square column bound at both points is 27.72—still conservative, but less so than the delta bound of 37.97. In contrast, the radial bound fails to upper bound the gradient error at these two points, illustrating that it may not be reliable when the columns of $\vect{U}$ are not orthogonal. Nonetheless, it remains significantly less conservative and much closer to the actual error compared to the other two bounds.
\end{Example}

\begin{table}[ht]
	\caption{Evaluation of truncation gradient error bounds in Example~\ref{ex:2}, with $\theta = 2/3$.} % title of Table
	\centering % used for centering table
	\begin{tabular}{c c c c c} % centered columns (4 columns)
		\hline %inserts double horizontal lines
		& $\|\err_t(\p_i)\|$ & $T_d$ & $T_c$ & $T_r$ \\ [0.3ex] % inserts table
		%heading
		\hline % inserts single horizontal line
		$\p_0$ & 6      & 6.1667 & 6.1667 & 6.1667 \\ 
		$\p_1$ & 6.3246 & 37,97 & 27.72  & 6.1667 \\
		$\p_2$ & 6.3246 & 37,97 & 27.72  & 6.1667 \\  % [1ex] adds vertical space
		\hline %inserts single line
	\end{tabular}
	\label{table:Ex2} % is used to refer this table in the text
\end{table}

Example \ref{ex:3} illustrates how the delta bound can become increasingly conservative compared to the square column and radial bounds as the dimension of the domain, $n_u$, increases.

\begin{Example} \label{ex:3}
	Consider the following set of interpolation points:
	\begin{align}
		& \p_0 = \vect{0}, \qquad \p_1 = 4\vect{e}_1, \qquad
		\p_j = \vect{e}_j, \quad j=2,\dots,n_u. \nonumber
	\end{align}
	
	With this configuration, the matrix $\vect{U}$ is orthogonal. As a result, the radial bound provides a valid  upper bound on the gradient error norm at $\p_0$. Observe that:
	\begin{align}
		& \bigl[ \ \tran{(\p_1-\p_0)}(\p_1-\p_0), \ \dots\ , \ \tran{(\p_{n_u}-\p_0)}(\p_{n_u}-\p_0) \ \bigr] = [4^2,\ 1,\ \dots\ ,\ 1], \nonumber
	\end{align}
	and that $\vect{U}^{-1} = \text{\rm diag}\{\frac{1}{4},\ 1,\ 1,\ \dots\ ,\ 1\}$. The corresponding radial bound is:
	\begin{align}
		& T_r = \frac{L}{2}\sqrt{4^2 + n_u - 1}. \nonumber
	\end{align} 
	
	On the other hand, since $\|\vect{U}^{-1}\| = 1$ and $\Delta = 4$ for all $n_u \geq 2$, the delta and square column bounds become:
	\begin{align}
		& T_d = \frac{L}{2} \cdot 4^2 \cdot \sqrt{n_u}, \nonumber\\
		& T_c = \frac{L}{2}\sqrt{4^4 + n_u - 1}. \nonumber
	\end{align}
	
	Figure~\ref{fig:ex3} plots the three bounds as functions of $n_u$. While all bounds increase with the domain dimension, the delta bound grows significantly faster than the others, quickly becoming overly conservative in higher dimensions. Among the three, the radial bound remains the least conservative throughout.
\end{Example}

\begin{figure}
	\centering \includegraphics[width=7cm]{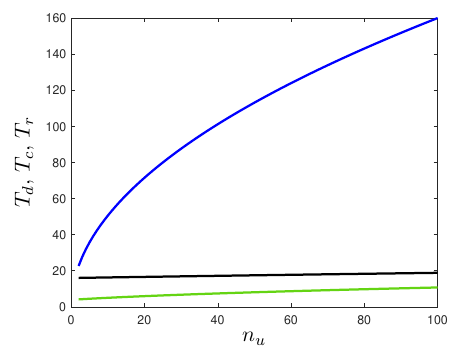} 
	
	\caption{Truncation gradient error bounds as a function of the dimension size in Example~\ref{ex:3}. \underline{Blue}: $T_d$. \underline{Black}: $T_c$. \underline{Green}: $T_r$.} 
	\label{fig:ex3}
\end{figure}

%%%%%%%%%%%%%%%%%%%%%%%%%%%%%%%%%%%%%%%%%
\subsection{Gradient Error due to Measurement Noise}
\label{ssec:noise}

Let us start by giving a geometrical characterization of the gradient error $\err_n$, which is defined in 
\eqref{equ:sop_s}. Let the vector $\en_{\text{x}} \in\Rn{\np+1}$ be normal to the $\np$-dimensional plane generated by the points
$(\tran{\p_0}, \ru_0)$, $(\tran{\p_1}, \ru_1)$, $\dots$, $(\tran{\p_{n_u}}, \ru_{n_u})$:
\begin{align}
	\left[ \begin{array}{cc}
		\tran{\p_1} - \tran{\p_0}, & \ru_1-\ru_0 \\
		\tran{\p_2} - \tran{\p_0}, & \ru_2-\ru_0  \\
		\vdots & \vdots  \\
		\tran{\p_{n_u}} - \tran{\p_0}, & \ru_{n_u}-\ru_0 \end{array} \right] \en_{\text{x}} = \vect{0}. \label{equ:normalx}
\end{align}

The normal vector $\en_{\text{x}} \in\Rn{\np+1}$ can be obtained, for instance, by singular value decomposition of the $\np\times (\np+1)$ matrix on the left side of \eqref{equ:normalx}, which is of rank $\np$ due to Assumption~\ref{assum:Poisedness}. Next, we define $\en_{\text{v}}$ as the unit vector that is normal to the points 
$(\tran{\p_0},\ 0)$, $(\tran{\p_1},\ 0)$, $\dots$, $(\tran{\p_{n_u}},\ 0)$, that is, $\en_{\text{v}}=\tran{(0,\ \dots\ ,\ 0,\ 1)}$. The angle $\alpha$ between $\en_{\text{x}}$ and $\en_{\text{v}}$ is given by
\begin{align}
	\alpha = \text{acos} \biggl( \frac{\tran{\en_{\text{v}}}\en_{\text{x}}}{\| \en_{\text{v}} \| \| \en_{\text{x}} \|} \biggr).  \label{equ:alfa}
\end{align}
In particular, by dividing each element of the vector $\en_{\text{x}}$ by the last one, $\en_{\text{x}}$ can be chosen as $\en_{\text{x}}=\tran{[\tran{\bar{\en}}\quad1]}$, $\bar{\en} \in\Rn{\np}$. With this
\begin{align}
	\alpha = \text{acos} \biggl( \frac{1}{\sqrt{\tran{\bar{\en}}\bar{\en}+1}} \biggr),  \label{equ:alfa2}
\end{align}
and from \eqref{equ:normalx}, we get
\begin{align}
	& \tran{\vect{U}} \bar{\en} + \tran{[\ 
		\ru_1 -\ru_0,\ \ru_2-\ru_0,\ \dots\ ,\  \ru_{n_u}-\ru_0\ ]} = \vect{0}. \label{eq:aux2}
\end{align}
It follows from \eqref{equ:sop_s} and \eqref{eq:aux2} that $\bar{\en} = - \err_n$, and \eqref{equ:alfa2} gives
\begin{align}
	\alpha = \text{acos} \Biggl( \frac{1}{\sqrt{\| \err_n \|^2+1}} \Biggr).  \label{equ:alfa3}
\end{align}

The vectors $\en_{\text{v}}$, $\en_{\text{x}}$ and $\err_n$ are represented in Figure~\ref{fig:nv} for the two-dimensional case ($\np=2$). Independently of the number of inputs, a plane $\mathcal{P}$ can be defined that contains the three vectors $\en_{\text{v}}$, $\en_{\text{x}}$ and $\err_n$. Notice that since $\err_n$ belongs to the input space, its component in the direction of $\ru$ is always zero. The relation \eqref{equ:alfa3}, which is represented in Figure \ref{fig:angulo}, takes place on the plane $\mathcal{P}$.
\noindent % Prevents indentation before the minipages
\begin{minipage}[t]{0.48\textwidth} % [t] for top alignment; 0.48\textwidth sets width
	\centering
	\begin{figure}[H] % [H] from float package can force placement here, or use [htbp]
		
		\includegraphics[width=7.5cm]{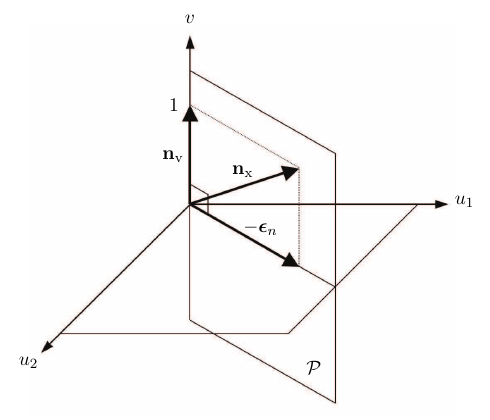} % Replace figure1.png with your first image file
		\caption{Geometrical representation of $\err_n$ in the two-dimensional case.\vspace{0.4cm}}
		\label{fig:nv} % Unique label for Figure 1
	\end{figure}
\end{minipage}
\hfill % Adds flexible horizontal space
\begin{minipage}[t]{0.48\textwidth} % [t] for top alignment; 0.48\textwidth sets width
	\centering
	\begin{figure}[H] % [H] from float package can force placement here, or use [htbp]
		
		\centering
		\includegraphics[width=4.5cm]{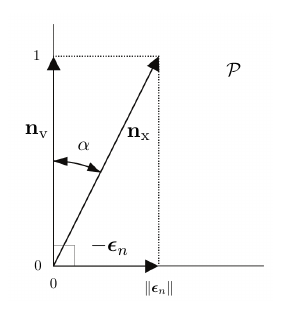} % Replace figure2.png with your second image file
		\caption{Geometrical representation of $\err_n$ on the plane $\mathcal{P}$.}
		\label{fig:angulo} % Unique label for Figure 2
	\end{figure}
\end{minipage}

Next, we consider two different upper bounds on the norm of the gradient error due to measurement noise.

\subsubsection{The Conditioning Bound}
\label{ssec:originoise}

We first consider the bound on $\|\err_n\|$ that was recently presented in \cite{Berahas:2022}. We will refer to this bound as the {\it conditioning bound}, and we will denote it by $N_c$.

\begin{Theorem} \label{th:distortion}
	Let Assumptions~\ref{assum:boundednoise} and \ref{assum:Poisedness} hold. Then,
	\begin{align}
		& \|\err_n\| \leq N_c := \frac{2\delta\sqrt{n_u}\|\hat{\vect{U}}^{-1}\|}{\Delta}, \label{eq:distortionbound}
	\end{align}
	with $\Delta$ and $\|\hat{\vect{U}}^{-1}\|$ as defined in \eqref{eq:deltadef}.
\end{Theorem} 

\begin{proof}
	From \eqref{equ:sop_s}, we have
	\begin{align}
		& \tran{\err_n}\vect{U} = [\ \ru_1-\ru_0, \quad \dots\ , \quad \ru_{n_u}-\ru_0\mbox{ }]. \nonumber
	\end{align}
	
	Hence, we can write the following component-wise inequality
	\begin{align}
		& |\tran{\err_n}\vect{U}| \leq 2\delta [\ 1,\ \dots\ , \ 1\ ]. \nonumber
	\end{align}
	
	Taking the norm gives
	\begin{align}
		& \|\tran{\err_n}\vect{U}\| \leq 2\delta\sqrt{n_u}. \nonumber
	\end{align}
	
	Hence,
	\begin{align}
		&  \|\err_n\| = \|\tran{\err_n}\vect{U}\vect{U}^{-1}\| \leq \|\tran{\err_n}\vect{U}\|\|\vect{U}^{-1}\| \leq  2\delta\sqrt{n_u}\|\vect{U}^{-1}\|, \nonumber
	\end{align}
	and \eqref{eq:distortionbound} follows from  $\hat{\vect{U}}^{-1} = \Delta\vect{U}^{-1}$.
\end{proof}

Note that the conditioning bound can be computed as $N_c = 2\delta\sqrt{n_u}\|\vect{U}^{-1}\|$, without evaluating the value of $\Delta$. However, writing the bound as in \eqref{eq:distortionbound} shows the effect that scaling of $\Delta$ has on the gradient error bound.

\begin{Remark}[Conditioning bound for FFD] \label{rmk:FFDconditioning}
	Let the sample points in $\mathcal{U}$ have the arrangement used in FFD with step size $h>0$. Then, the gradient-error bound in Theorem~\ref{th:distortion} reduces to
	\begin{align}
		& \|\err_n\| \leq  \frac{2\delta\sqrt{n_u}}{h}. \label{eq:distortionboundFFD}
	\end{align}
\end{Remark}

\subsubsection{The L-min Bound}
\label{ssec:lmin_bound}

In the case of bounded measurement noise, the gradient-error norm $\|\err_n\|$ is bounded by the largest possible gradient-error norm, corresponding to the worst-case scenario for the measurement noise,
\begin{align}
	& \|\err_n\| \leq \|\err_n\|_{\rm max}, \nonumber
\end{align}
with
\begin{align}
	\|\err_n\|_{\rm max} = \max_{\vect{v}} \quad & \|\err_n\|  \label{eq:errn_max}\\
	\text{\rm s.t.} \quad & -\delta \leq \nu_j \leq \delta, \quad j = 0,1,\dots,n_u, \nonumber
\end{align}
with the decision variables $\vect{v}=\tran{[\ru_0,\  \ru_1,\ \dots\ ,\ \ru_{\np}]}$.
The approach for computing $\|\err_n\|_{\rm max}$ was presented in \cite{Marchetti:2010}. Since only a sketch of the proof is given in \cite{Marchetti:2010}, the complete proof is provided in Theorem~\ref{th:lmin} below. The concept of complement affine subspaces is required, for which a detailed description is given in Appendix~\ref{app:affine_subspaces}.

\begin{Theorem} \label{th:lmin}
	Let Assumptions~\ref{assum:boundednoise} and \ref{assum:Poisedness} hold. Then,
	\begin{align}
		& \|\err_n\| \leq N_l = \frac{2\delta}{l_{\text{\rm min}}}, \label{eq:lminbound}
	\end{align}
	where $N_l := \|\err_n\|_{\rm max}$, and $l_{\text{\rm min}}$ is the shortest distance between all the complement affine subspaces that can be generated from the sample points in $\mathcal{U}$.
\end{Theorem}

\begin{proof}
	The proof includes two steps.
	
	\textbf{1.} 
	Maximizing $\|\err_n\|^2$ instead of $\|\err_n\|$ does not change the optimal solution of Problem~\eqref{eq:errn_max}. From \eqref{equ:sop_s}, we have
	\begin{align}
		& \|\err_n\|^2 =  [\ \ru_1-\ru_0, \quad \dots\ , \quad \ru_{n_u}-\ru_0\mbox{ }]\ \vect{U}^{-1}\tran{[\vect{U}^{-1}]}\tran{[\ \ru_1-\ru_0, \quad \dots\ , \quad \ru_{n_u}-\ru_0\mbox{ }]}. \nonumber 
	\end{align}
	
	Upon defining the variables $x_i = \ru_i - \ru_0$, for $i=1,\dots,n_u$, Problem~\eqref{eq:errn_max} can be reformulated as
	\begin{align}
		\|\err_n\|_{\rm max}^2 = \max_{\vect{x}} \quad & \tran{\vect{x}}\vect{R}\ \vect{x} \label{eq:worstcase_opt2}
		\\
		\text{s.t.} \quad & -2\delta \leq x_i \leq 2\delta, \quad i=1,\dots,\np, \nonumber
	\end{align}
	where $\vect{R}=\vect{U}^{-1}\tran{[\vect{U}^{-1}]}$ is positive definite due to Assumption~\ref{assum:Poisedness}.  It is well known that a global optimum of Problem \eqref{eq:worstcase_opt2} is reached at the intersection of the input bounds (see e.g. Theorem~32.1 in \cite{Rockafellar:97}).
	This means that, at the solution to Problem~\eqref{eq:worstcase_opt2}, there is a proper subset of points
	$\mathcal{U}_A \subsetneq \mathcal{U}$ for which the error is $\ru_A = \delta$, and a complement subset of points $\mathcal{U}_C := \mathcal{U}\setminus\mathcal{U}_A$, for which the error is $\ru_C = -\delta$ (or  $\ru_A = -\delta$ and $\ru_C = \delta$). 
	
	\textbf{2. }Next, we show that, if for any pair of complement affine subspaces, all the points in $\mathcal{U}_A$ have the error $\ru_A =\delta$ and all the points in $\mathcal{U}_C$ have the error $\ru_C = -\delta$ (or $-\delta$ and $\delta$ respectively), then the error vector $\err_n$ is normal to both complement affine subspaces. Let us consider the sets $\mathcal{U}_A = \{\p_0, \p_1,\dots,\p_{\np^A-1}\}$ and $\mathcal{U}_C = \{\p_{\np^A},\dots,\p_{\np}\}$.
	Since the simplex gradient is independent of the basis used, the gradient error $\err_n$ in \eqref{equ:sop_s} can also be written as
	\begin{align}
		& \tran{\err_n} = [\ \ru_1-\ru_0, \ \dots\ , \ \ru_{\np^A-1}-\ru_0,\  \ru_{\np^A+1}-\ru_{\np^A}, \ \dots\ , \ \ru_{\np}-\ru_{\np^A},\ \ru_{\np^A-1}-\ru_{\np^A} \ ] \ \vect{U}^{-1}_{AC} \label{equ:sop_s_proof} 
	\end{align}
	with 
	\begin{align} 
		& \vect{U}_{AC} = [\ \p_1-\p_0,\ \dots\ , \ \p_{\np^A-1}-\p_0,\ \ \p_{\np^A+1}-\p_{\np^A},\ \dots\ , \p_{\np}-\p_{\np^A},\ \p_{\np^A-1}-\p_{\np^A} \ ].
	\end{align}
	
	Taking the error $\ru_A =\delta$ for all the points in $\mathcal{U}_A$ and the error error $\ru_C = -\delta$ for all the points in $\mathcal{U}_C$, equation \eqref{equ:sop_s_proof} becomes
	\begin{align}
		& \tran{(\err_n^{AC})} = [\ 0, \ \dots\ , \ 0,\ 2\delta \ ] \ \vect{U}^{-1}_{AC}. \label{equ:sop_s_proof2} 
	\end{align}

	Using the matrix $\vect{Q}$ defined in \eqref{equ:calcnormal} in $\vect{U}_{AC}$, \eqref{equ:sop_s_proof2} becomes:
	\begin{align}
		 \left[ \begin{array}{c}
				\vect{Q} \\
				\tran{(\p_{\np^A-1}- \p_{\np^A})} \end{array} \right] \err_n^{AC}  = \left[ \begin{array}{c} {\bf 0}_{\np-1} \\ 2\delta \end{array} \right], \label{equ:Uac2}
	\end{align}
	from where $\vect{Q}\err_n^{AC} = \vect{0}$, and $\err_n^{AC}$ is orthogonal to both complement affine subspaces.
	
	From \eqref{equ:Uac2} and \eqref{equ:calcdist}, one has
	\begin{align}
		& \bigl\lvert \tran{(\err_n^{AC})} (\p_{\np^A-1} - \p_{\np^A})\bigr\rvert = l_{AC}\| \err_n \|_{AC} =  2\delta, \nonumber\\
		& \| \err_n \|_{AC} = \frac{2\delta}{l_{AC}}, \label{equ:Uac3}
	\end{align}
	where $l_{AC}$ is the distance between the two complement affine subspaces. 
	From all possible complement subsets $\mathcal{U}_A$ and $\mathcal{U}_C$, $\| \err_n \|_{\text{max}}$ occurs for the complement subsets that correspond to the nearest complement affine subspaces, that is, for $l_{AC}=l_{\text{min}}$. It follows that $\| \err_n \|_{\text{max}} = \frac{2\delta}{l_{\text{min}}}$, which completes the proof.
\end{proof}
\noindent % Prevents indentation before the minipages
\begin{minipage}[t]{0.52\textwidth} % [t] for top alignment; 0.48\textwidth sets width
	\centering
	\begin{figure}[H] % [H] from float package can force placement here, or use [htbp]
		
		\includegraphics[width=8.5cm]{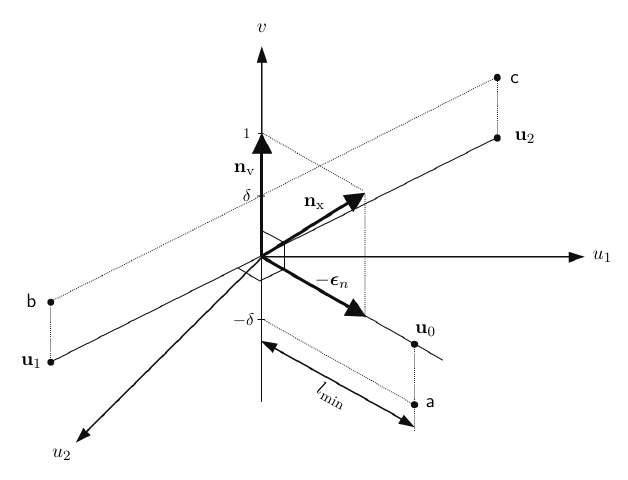} % Replace figure1.png with your first image file
		\caption{Geometrical representation of worst-case scenario for measurement noise in the two-dimensional case.\vspace{0.4cm}}
		\label{fig:prop2} % Unique label for Figure 1
	\end{figure}
\end{minipage}
\hfill % Adds flexible horizontal space
\begin{minipage}[t]{0.44\textwidth} % [t] for top alignment; 0.48\textwidth sets width
	\centering
	\begin{figure}[H] % [H] from float package can force placement here, or use [htbp]
		
		\centering
		\includegraphics[width=6cm]{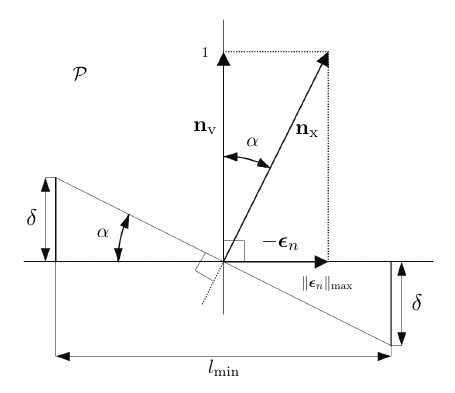} % Replace figure2.png with your second image file
		\caption{Geometrical representation of worst-case scenario for measurement noise on the plane $\mathcal{P}$.}
		\label{fig:worstcase} % Unique label for Figure 2
	\end{figure}
\end{minipage}

\vspace{0.3cm}
The limiting situation given by Theorem~\ref{th:lmin} is represented in Figure~\ref{fig:prop2} for the two-dimensional case $(\np=2)$. 
There are three distances between complement affine subspaces that can be evaluated, namely,  $l_1$: the distance between $\p_1$ and the line generated by $\p_0$ and $\p_2$; $l_2$: the distance between $\p_2$ and the line generated by $\p_0$ and $\p_1$; and $l_3$: the distance between $\p_0$ and the line generated by $\p_1$ and $\p_2$.
Figure~\ref{fig:prop2} considers the case where $l_{\text{min}} = l_3 = \min\{ l_1,l_2,l_3 \}$.
When the measurement error is $\delta$ for $\p_1$ and $\p_2$, and $-\delta$ for $\p_0$, then $\err_n$ is normal to both complement affine subspaces. According to Theorem~\ref{th:lmin}, this situation leads to the norm of $\err_n$ taking its largest possible value $\| \err_n \|_\text{max}$.
Recall that the norm of $\en_{\text{v}}$ is 1 by definition and vector $\en_{\text{x}}$ is normal to the plane generated by the points $\matr{a}$, $\matr{b}$ and $\matr{c}$ (see \eqref{equ:normalx}). 

Independently of the number of inputs, a plane $\mathcal{P}$ can be defined, which contains the three vectors $\en_{\text{v}}$, $\en_{\text{x}}$ and $\err^n$. The representation of the worst-case scenario of Theorem~\ref{th:lmin} on the plane $\mathcal{P}$ is represented in Figure~\ref{fig:worstcase}. Notice that $2\delta = l_{\text{min}} \text{\rm tan}(\alpha)$, and $\| \err_n \|_\text{max} =  \text{\rm tan}(\alpha)$. Hence, it is verified that $\| \err_n \|_{\text{max}} = \frac{2\delta}{l_{\text{min}}}$.

\begin{Remark}[Least upper bound] \label{rmk:leastupperbound}
	Note that $N_l = \|\err_n\|_{\rm max}$ implies that the l-min bound $N_l$ is the least upper bound on the gradient error due to bounded measurement noise. This means that the conditioning bound $N_c$ can only be greater than or equal to $N_l$ for any poised sample set $\mathcal{U}$.
\end{Remark}

\begin{Remark}[L-min bound for FFD] \label{rmk:FFDlmin}
	In the FFD arrangement with step size $h>0$, the shortest distance between complement affine subspaces corresponds to the distance between $\p_0$ and the hyperplane generated by the points $\p_j = \p_0 + h\vect{e}_j$, $j=1,\dots,n_u$. This distance is $l_{min} = h/\sqrt{n_u}$. Hence, the gradient-error bound in Theorem~\ref{th:lmin} becomes
	\begin{align}
		& \|\err_n\| \leq  \frac{2\delta\sqrt{n_u}}{h}. \label{eq:lminboundFFD}
	\end{align}
\end{Remark}

Remarks \ref{rmk:FFDconditioning} and \ref{rmk:FFDlmin} show that the conditioning bound and the l-min bound are both the same in the case of the FFD arrangement. However, the conditioning bound becomes more conservative than the l-min bound for other geometrical arrangments of the sample points, as will be illustrated in Section~\ref{sec:duality}.

%%%%%%%%%%%%%%%%%%%%%%%%%%%%%%%%%%%%%%%%%
\subsection{Total Gradient Error}
\label{ssec:totalerror}

The upper bounds on the gradient-error norms due to truncation noise and measurement noise can readily be used to bound the total gradient error. 
For instance, in Theorem~\ref{th:totalsquare}, an upper bound $E_c$ on the norm of the gradient error at $\p_0$, is obtained by combining the square column bound and the l-min bound.

\begin{Theorem} \label{th:totalsquare}
	Let Assumptions~\ref{assum:boundednoise}, \ref{assum:differ}, \ref{assum:Lipschitz}, and \ref{assum:Poisedness} hold. Then,
	\begin{align}
		& \|\vect{g} - \nabla f(\p_j)\| \leq E_c := T_c + \frac{2\delta}{l_{min}}.
	\end{align}
\end{Theorem}

\begin{proof}
	The proof follows directly by combining \eqref{eq:normineq}, \eqref{eq:colnormbound}, and \eqref{eq:lminbound}.
\end{proof}

\vspace{0.5cm}

In the case of FFD, an upper bound on the gradient-error norm is given by
\begin{align}
	& E_{FFD} := \frac{L\sqrt{n_u}\ h}{2} + \frac{2\delta\sqrt{n_u}}{h}, \label{eq:EFFD}
\end{align}
which is a well-known bound for the FFD approximation to the gradient $\nabla f(\p_0)$ \cite{Berahas:2022}. Note that this bound is not improved by the use of the square column bound and the l-min bound.
It is easy to show that the step size $h$ that minimizes $E_{FFD}$ is $h^\star = 2\sqrt{\delta/L}$.

%%%%%%%%%%%%%%%%%%%%%%%%%%%%%%%%%%%%%%%%%%%%%%%%%%%%%%%%%%%%%%%%%%%%%%%%%%%%%
\section{Approximate Gradient-Error Bounds} \label{sec:approx}

The delta bound and the square column bound are proven upper bounds on the simplex gradient error due to truncation. However, these guarantees often come at the cost of excessively conservative estimates.

Excessively conservative gradient error bounds in DFO can slow convergence due to overly cautious steps. They may lead to the rejection of productive search directions, excessive sampling, inefficient use of evaluations, and even premature termination. As a result, the optimization process becomes unnecessarily costly and less effective.

To address this, we introduce the notion of an {\it approximate gradient error bound}—a measure that aims to upper bound the gradient error norm without the strict conservatism of classical bounds. Unlike exact bounds, an approximate gradient error bound is not guaranteed to hold in all cases, but is designed to be valid in most practical scenarios. The goal is to provide a tighter, more realistic assessment of gradient accuracy while still retaining the structure and interpretability of a bound. Crucially, it offers a practically useful tool for controlling gradient accuracy in DFO, without incurring the drawbacks associated with excessively conservative error bounds.

\subsection{Radial Bound as an Approximate Gradient Error Bound}

We now propose the radial bound, defined in \eqref{eq:radialbound}, as an approximate gradient error bound for use in DFO. This bound is particularly appealing due to the following properties:
\begin{itemize}
	\item The radial bound upper bounds the projection of the gradient error along all directions $\p_j-\p_0$, given by the columns of $\vect{U}$. When these directions are orthogonal, the radial bound provides a strict upper bound on the gradient error norm.
	\item When the columns of $\vect{U}$ are not orthogonal, the radial bound still frequently upper bounds the gradient error norm, even though this is not formally guaranteed.
	\item The radial bound generalizes easily to higher dimensions, as it depends only on the gradient Lipschitz constant and the radius defined by the interpolation points. 
	\item The radial bound produces a uniform value when evaluated at any of the $(n_u + 1)$ points in the interpolation set. This aligns with the fact that the simplex gradient yields the same estimate at any vertex of the simplex. Given that the gradient Lipschitz constant is the only information available, there is no reason to expect the gradient error to be significantly larger or smaller at one vertex compared to another. Yet, the delta and square column bounds may assign significantly different upper bounds to different vertices of the simplex.
\end{itemize}

In addition to being conservative, the delta bound \eqref{eq:deltabound} involves a subtle tradeoff between the conditioning of the interpolation points—measured by $\|\hat{\vect{U}}^{-1}\|$ and their mutual proximity, characterized by the value of $\Delta$. This tradeoff can be difficult to perceive or interpret in practice. In contrast, the radial bound encodes this relationship more transparently through the radius of the sphere defined by the interpolation points, as illustrated in Example~\ref{ex:circle}.

\begin{Example} \label{ex:circle}
	Consider a function $f(\p)$, $\p\in\mathbb{R}^2$, with gradient Lipschitz constant $L$. The computation of the simplex gradient requires an interpolation set of three points. Taking the green point as the reference, both the left and right plots in Figure~\ref{fig:circulo} depict sets of points with the same value of $\Delta$. However, the left set is poorly conditioned and spans a large circle, while the right set is well conditioned and spans a much smaller one. As a result, the radial bound $T_r = L r$ is significantly smaller for the right configuration.
\end{Example}
\begin{figure}
	
	\centering \includegraphics[width=7cm]{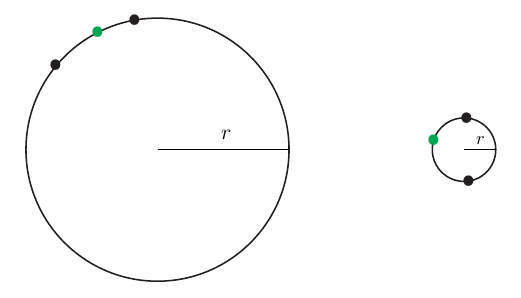} 
	
	\caption{Radial bound: Interplay between the conditioning of the interpolation points and the distance between them. The points on both circles have the same value of $\Delta$ with respect to the green point, but different conditioning.} 
	\label{fig:circulo}
\end{figure}

From the delta bound, one might infer that improving gradient accuracy simply requires reducing $\Delta$, i.e., bringing the points in the interpolation set closer together. This inference is valid only when the geometry of the sample set is preserved, such as by scaling $\vect{U}$ uniformly. However, when the geometry is allowed to vary, this reasoning can be misleading: the gradient error may grow arbitrarily large for any fixed value of $\Delta$, depending on the conditioning of the interpolation matrix. In contrast, the radial bound—by depending solely on the radius of the sphere defined by the sample points—captures both proximity and geometry in a unified and easily interpretable way.\\

The following example illustrates the suitability of the radial bound as an approximate gradient upper bound for a function in $\mathbb{R}^5$. 

\begin{Example} \label{ex:Ex5}
	Consider the non-convex function
	\begin{align}
		& f(\p) = u_1^2 - u_2^2 + (u_3-u_4)^2 - u_5^2 + u_1 u_3 - u_2 u_4 - 6 u_5 + 5 u_2, \nonumber
	\end{align}
where $\p = \tran{[u_1, \dots, u_5]}$. This function has a constant Hessian, and its gradient Lipschitz constant is $L = 4.3014$, computed as the spectral radius of the Hessian.

To assess the behavior of the gradient error and the gradient error bounds under diverse geometric configurations, we generate 1000 random interpolation sample sets $\mathcal{U}$, each consisting of six points $(\p_0,\p_1,\dots,\p_5)$ selected uniformly within the box $\mathcal{B} = \{\p\in\mathbb{R}^5\ :\ -1\leq u_i \leq 1,\ \text{for}\ i=1,\dots,5\}$. Sample sets with $\|\vect{U}^{-1}\| > 8$ are discarded to avoid extreme amplification of the gradient error due to poor conditioning.

For each sample set, we evaluate the true norm of the truncation gradient error, $\|\vect{\epsilon}_t(\p_0)\|$, along with the square column bound and the radial bound. The results are shown in Figure~\ref{fig:Ejemplo_5inputs}. The square column bound (purple) consistently lies well above both the actual gradient error norm (black) and the radial bound (red), reflecting its conservative nature. In contrast, the radial bound successfully upper bounds the gradient error norm in 87.8\% of the cases, and provides a much tighter and more informative approximation in general.
\end{Example}

\begin{figure}	
	\centering \includegraphics[width=7.5cm]{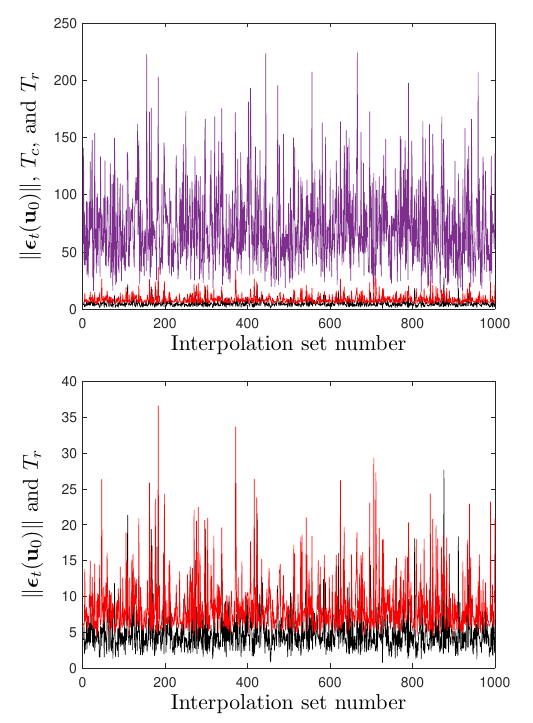} 
	
	\caption{Gradient error norm and gradient error bounds for 1000 random interpolation sets in Example~\ref{ex:Ex5}. \underline{Black}: $\|\vect{\epsilon}_t(\p_0)\|$; \underline{Purple}: $T_c$; \underline{Red}: $T_r$.} 
	\label{fig:Ejemplo_5inputs}
\end{figure}

\subsection{Extension of the Radial Bound to Non-Vertex Points.}

We aim to obtain approximate gradient error bounds at non-vertex points lying within the simplex (or convex hull) defined by the interpolation points. %Extended radial bound.

Following the same reasoning as in \eqref{eq:deltabound_u2}, a pointwise approximate bound $T_f$ on the gradient-error norm can be obtained from the radial bound as
\begin{align}
	& T_f(\p) := T_r + L\|\p-\p_0\|. \label{eq:radialbound_u2}
\end{align}

However, when this bound is evaluated at a vertex $\p_j\in\mathcal{U}$, $j\neq 0$, we obtain $T_f(\p_j) := T_r + L\|\p_j-\p_0\| > T_r$, which contradicts the uniformity of the radial bound across all interpolation points, as established in Remark~\ref{rmk:radialuniform}.   
Moreover, empirical observations indicate that the gradient error tends  to decrease—rather than increase—at points located within the interior of the convex hull defined by the interpolation set. 

Our goal is therefore to construct an approximate bounding function for the gradient error norm that:
\begin{enumerate}
	\item Returns the value $T_r$ at all interpolation points $\p_j\in\mathcal{U}$, and 
	\item More accurately reflects the behavior of the gradient error within the convex hull of the interpolation set.
\end{enumerate} 

To explore this, we begin by examining the one-dimensional case $(n_u = 1)$. Let us define the functions
\begin{align}
	& q(\pk) = \frac{L}{2}(\pk - \pk_0)^2, \nonumber
\end{align}
and
\begin{align}
	& h(\pk) = q(\pk) + f(\pk_0) + \lambda(\pk - \pk_0), \nonumber
\end{align}
with
\begin{align}
	& \lambda = \frac{f(\pk_1)-f(\pk_0)}{\pk_1-\pk_0} - \frac{q(\pk_1)-q(\pk_0)}{\pk_1-\pk_0}. \nonumber
\end{align}

It is easy to show that $h(\pk_0)=f(\pk_0)$ and $h(\pk_1)=f(\pk_1)$. Hence, $h(\pk)$ is the quadratic function with second derivative $\nabla^2h=L$ that interpolates $f$ at $\pk_0$ and $\pk_1$.
The gradient error of the function $h$ is defined as
\begin{align}
	& \epsilon_t^h(\pk) = \frac{f(\pk_1)-f(\pk_0)}{\pk_1 - \pk_0} - \nabla h(\pk). \nonumber
\end{align}

Noting that $\nabla h(\pk) = L(\pk- \pk_0) + \lambda$, we get
\begin{align}
	& \epsilon_t^h(\pk) = \frac{q(\pk_1)-q(\pk_0)}{\pk_1 - \pk_0} - L(\pk - \pk_0) = \frac{L}{2}(\pk_1 - \pk_0) - L(\pk - \pk_0) = L(\pk_c - \pk), \nonumber
\end{align}
with the center point $\pk_c = (\pk_0+\pk_1)/2$. Let us define the {\it extended radial bound} as
\begin{align}
	& T_h(\pk) := |\epsilon_t^h(\pk)| = L|\pk_c - \pk|. \nonumber
\end{align}

Note that $T_h(\pk_0) = T_h(\pk_1) = L r = T_r$, and $T_h(\pk_c) = 0$.
The gradient error of the function $f$ is defined as
\begin{align}
	& \epsilon_t^f(\pk) = \frac{f(\pk_1)-f(\pk_0)}{\pk_1 - \pk_0} - \nabla f(\pk). \nonumber
\end{align}

We know from the mean-value theorem that there exists $\bar{\pk}\in(\pk_0,\pk_1)$ such that $\epsilon_t^f(\bar{\pk}) = 0$.\\

The following example illustrates the behavior of the gradient error $|\epsilon_t^f(\pk)|$, the approximate gradient-error bound $T_f(\pk) = T_r + L|\pk-\pk_0|$, and the extended radial bound $T_h(\pk)$, for the case of the exponential function.

\begin{Example} \label{exm:Th}
	Let us consider the function $f(\pk) = e^u$ in the interval $\mathcal{Q} = [0.5,2.5]$. The gradient Lipschitz constant $L$ is an upper bound on the second derivative of $f$ on $\mathcal{Q}$, that is, $L=e^{2.5}$. The gradient (derivative) of $f$ is approximated by finite differences from the points $(\pk_0,f(\pk_0))$ and $(\pk_1,f(\pk_1))$, with $\pk_0=1$ and $\pk_1=2$. The quadratic function $h(\pk)$ interpolates $f$ at $\pk_0$ and $\pk_1$, and has a conservative curvature with respect to $f$, as shown in the top plot of Figure~\ref{fig:fun1}. The gradient-error $|\epsilon_t^f(\pk)|$ is compared with $T_f(\pk)$ in the middle plot of Figure~\ref{fig:fun1}, while the comparison with the extended radial bound $T_h(\pk)$ is provided in the bottom plot of Figure~\ref{fig:fun1}. This last plot shows that $T_h(\pk)$ is not an upper bound on $|\epsilon_t^f(\pk)|$ for all $\pk\in[\pk_0,\pk_1]$. Indeed, $T_h(\pk_c) = 0$, while $|\epsilon_t^f(\pk_c)|>0$. However, $T_h(\pk)$ is equal to the radial bound $T_r$ at the interpolation points, and gives a better approximation of the gradient error in the interval $[\pk_0,\pk_1]$.
\end{Example}

\begin{figure}[h]
	
	\centering \includegraphics[width=5.5cm]{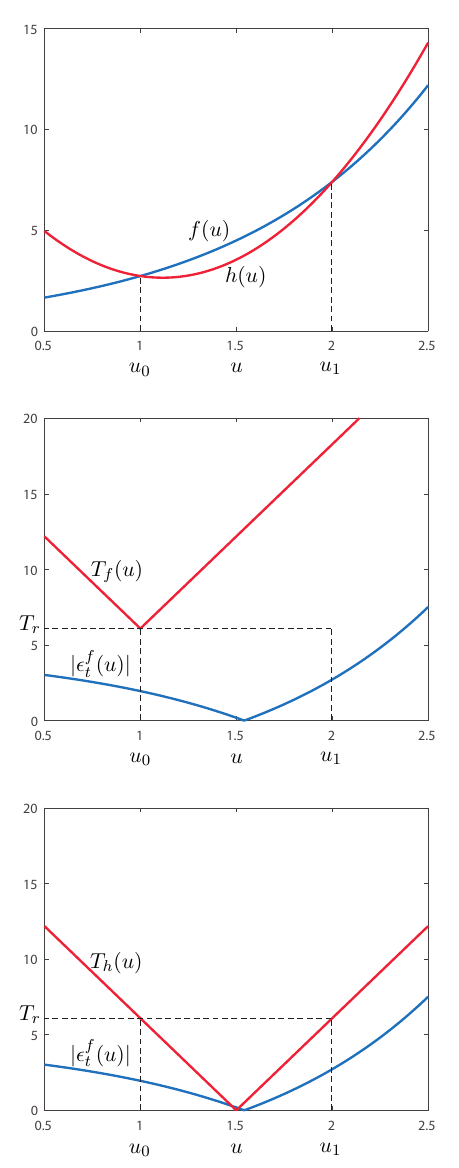} 
	
	\caption{Illustration of the gradient error $|\epsilon_t^f(\pk)|$, the approximate bound $T_f(u)$, and the extended radial bound $T_h(u)$ in Example~\ref{exm:Th}.} 
	\label{fig:fun1}
\end{figure}

The definition of the bound $T_h(\pk)$ can be extended to the $n_u$-dimensional case. Consider the function
\begin{align}
	& q(\p) = \frac{1}{2}\tran{(\p-\p_0)}\vect{D}(\p-\p_0), \nonumber
\end{align}
with $\vect{D} = L\vect{I}_{n_u}$, and
\begin{align}
	& h(\p) = q(\p) + f(\p_0) + \tran{\vect{\lambda}}(\p-\p_0), \nonumber
\end{align}
with
\begin{align}
	& \tran{\vect{\lambda}} = \tran{(\vect{y}_f - \vect{y}_q)}\vect{U}^{-1}, \nonumber\\
	& \tran{\vect{y}_f} = \left[ f(\p_1)-f(\p_0),\ \dots\ ,\ f(\p_{n_u})-f(\p_0)\right], \nonumber\\
	& \tran{\vect{y}_q} = \left[ q(\p_1)-q(\p_0),\ \dots\ ,\ q(\p_{n_u})-q(\p_0)\right]. \nonumber
\end{align}

Noting that $f(\p_0) + \tran{\vect{\lambda}}(\p-\p_0)$ is the linear interpolation model for $f(\p)-q(\p)$, it follows that $f(\p_0) + \tran{\vect{\lambda}}(\p_j-\p_0) = f(\p_j)-q(\p_j)$, for all $\p_j\in\mathcal{U}$. Hence, $h(\p_j) = f(\p_j)$, for $j=0,1,\dots,n_u$, and $h(\p)$ is the quadratic function with Hessian matrix $\vect{D} = L\vect{I}_{n_u}$ that interpolates $f(\p)$ at all the points in the interpolation set $\mathcal{U}$.

The gradient error of function $h$ is
\begin{align}
	\tran{\err_t^h(\p)} &= \tran{\vect{g}_t} - \tran{\nabla h(\p)} 
	= \tran{\vect{y}_f}\vect{U}^{-1} - \tran{(\p-\p_0)}\vect{D} - \tran{\vect{\lambda}}
	= \tran{\vect{y}_q}\vect{U}^{-1} - L\tran{(\p-\p_0)},  \nonumber
\end{align}
or
\begin{align}
	\tran{\err_t^h(\p)} &= \frac{L}{2}\left[\tran{(\p_1-\p_0)}(\p_1-\p_0),\ \dots\ ,\ \tran{(\p_{n_u}-\p_0)}(\p_{n_u}-\p_0) \right]\vect{U}^{-1} 
	- L\tran{\p} + L\tran{\p_0}.  \label{eq:approxA}
\end{align}

In turn, we can write
\begin{align}
	L\tran{\p_0} &= \frac{L}{2}2\tran{\p_0}\vect{U}\vect{U}^{-1}  
	= \frac{L}{2}\left[2\tran{\p_0}(\p_1-\p_0),\ \dots\ ,\ 2\tran{\p_0}(\p_{n_u}-\p_0) \right]\vect{U}^{-1},  \label{eq:approxB}
\end{align}
and
\begin{align}
	& \tran{(\p_j-\p_0)}(\p_j-\p_0) + 2\tran{\p_0}(\p_j-\p_0) = \tran{\p_j}\p_j - \tran{\p_0}\p_0, \qquad \text{for}\ \ j=1,2,\dots,n_u. \label{eq:approxC}
\end{align}

From \eqref{eq:approxA}, \eqref{eq:approxB}, and \eqref{eq:approxC}, we have
\begin{align}
	\tran{\err_t^h(\p)} &= \frac{L}{2}\left[\tran{\p_1}\p_1 - \tran{\p_0}\p_0,\ \dots\ ,\ \tran{\p_{n_u}}\p_{n_u} - \tran{\p_0}\p_0 \right]\vect{U}^{-1} 
	- L\tran{\p}  \nonumber\\
	&= L(\tran{\p_c} - \tran{\p}) \label{eq:approxD}.
\end{align}

We now define the extended radial bound as
\begin{align}
	& T_h(\p) := \|\err_t^h(\p)\| = L\|\p_c - \p\|. \label{eq:approxbound}
\end{align}

Clearly, $T_h(\p_j) = Lr = T_r$, for all $\p_j\in\mathcal{U}$, and $T_h(\p_c) = 0$.

\subsection{The Simplex Bound}
\label{ssec:simplex}

Instead of focusing on the gradient accuracy at $\p_0$, or at the points in the sample set $\mathcal{U}$, we argue that, for the purpose of optimization, it may be sufficient to pay attention to gradient accuracy at any point that belongs to the convex hull determined by the points in $\mathcal{U}$. This convex hull, which we shall denote by $\text{\rm conv}(\mathcal{U})$, corresponds to the polyhedral set that has $\p_0,\p_1,\dots,\p_{n_u}$ as its extreme points.
Based on this premise, we define the {\it simplex bound} $T_s$ as the minimum value that the extended radial bound $T_h(\p)$ takes in $\text{\rm conv}(\mathcal{U})$.
This approximate gradient error bound was proposed in \cite{Marchetti:2013a} in the context of dual real-time optimization.
\begin{align}
	& T_s := T_h(\p_s^\star), \label{eq:Ts}
\end{align}
with
\begin{align}
	& \p_s^\star = [\p_0,\ \p_1,\ \dots,\ \p_{n_u}]\vect{d}^\star, \qquad \text{with}\ \vect{d}^\star\in\mathbb{R}^{n_u+1}, \label{eq:usstar}
\end{align}
where
\begin{align}
	\vect{d}^\star = \text{arg}\min_{\vect{d}} \quad & \tran{(\p_c - \p_s)}(\p_c - \p_s) \label{eq:Probsimplex}\\
	\text{s.t.} \quad & \p_s = [\p_0,\ \p_1,\ \dots,\ \p_{n_u}]\vect{d} \nonumber\\
	& \vect{d} \geq \vect{0} \nonumber\\
	& \sum_{i=1}^{n_u+1} d_i = 1. \nonumber
\end{align}

The constraints in Problem~\eqref{eq:Probsimplex} ensure that $\p_s$ belongs to $\text{\rm conv}(\mathcal{U})$. Hence, $\p_s^\star$ is the point in the polyhedral set that is closest to the center point $\p_c$, and therefore, $\p_s^\star$ minimizes the extended radial bound $T_h(\p)$ subject to $\p\in\text{\rm conv}(\mathcal{U})$. Note that, if $\p_c\in\text{\rm conv}(\mathcal{U})$, then $\p_s^\star = \p_c$, and $T_s=0$. However, if $\p_c\notin\text{\rm conv}(\mathcal{U})$, then $T_s>0$.

%%%%%%%%%%%%%%%%%%%%%%%%%%%%%%%%%%%%%%%%%%%%%%%%%%%%%%%%%%%%%%%%%%%%%%%%%%%%%

\section{Comparison of Gradient-Error Bounds}
\label{sec:duality}

If the reference point $\p_0$ is replaced by a variable input $\p$, all the gradient error bounds introduced in Sections~\ref{sec:errorbounds} and \ref{sec:approx} can be formulated as functions of $\p$. Given the $n_u$ sample points $(\p_1,f(\p_1)), \dots, (\p_{n_u},f(\p_{n_u}))$, the radial bound $T_r(\p)$, for example, provides an approximate upper bound on the norm of the truncation gradient error that would result from evaluating $f$ at $\p$, and then computing the simplex gradient using the points $(\p,f(\p)), (\p_1,f(\p_1)), \dots, (\p_{n_u},f(\p_{n_u}))$. A desired upper bound $E^U$ on the gradient error can then be enforced by imposing a constraint of the form:
\begin{align}
	& \text{\rm a)} \quad T_r(\p) \leq E^U. \label{eq:constr_Tr}\\
	& \text{\rm b)} \quad N_l(\p) \leq E^U. \label{eq:constr_Nl}\\
	& \text{\rm c)} \quad E_r(\p) = T_r(\p) + N_l(\p) \leq E^U. \label{eq:constr_TrNl}
\end{align}

Constraint \eqref{eq:constr_Tr} uses the radial bound to control the gradient error due to truncation, \eqref{eq:constr_Nl} applies the l-min bound to limit the gradient error due to measurement noise, and \eqref{eq:constr_TrNl} combines both to control the total gradient error.

In general, the gradient error constraints introduced above define two disjoint feasible regions, located on opposite sides of the hyperplane defined by the points $\p_1,\dots,\p_{n_u}$. To characterize this hyperplane, we introduce the matrix
\begin{align}
	& \vect{U}_s = \left[\p_2-\p_1,\ \p_3-\p_1,\dots,\p_{n_u}-\p_1\right] \in \mathbb{R}^{n_u\times(n_u-1)}. \nonumber
\end{align}
Assuming the columns of $\vect{U}_s$ are linearly independent, they span a unique hyperplane containing the $n_u$ points $\p_1,\dots,\p_{n_u}$. Let $\vect{n}\in\mathbb{R}^{n_u}$ be a unit normal vector to this hyperplane. Then, $\tran{(\vect{U}_s)}\vect{n}=\vect{0}$, and the hyperplane can be written as $\mathcal{H} = \{\p\in\mathbb{R}^{n_u}:\ \tran{\vect{n}}\p=b,\ \text{\rm where}\ b=\tran{\vect{n}}\p_1\}$.
It is important to note that if $\p\in\mathcal{H}$, then the set of $n_u+1$ points $\{\p,\p_1,\dots,\p_{n_u}\}$ will lie on the same hyperplane and therefore will not be poised for linear interpolation in $\mathbb{R}^{n_u}$.\\

In this section, we compare some of the gradient-error bounds introduced in Sections~\ref{sec:errorbounds} and \ref{sec:approx} by analyzing the feasible regions generated by the corresponding gradient error constraints in the two-dimensional case.

\begin{Example}[Gradient error due to truncation] \label{ex:truncationerror}
	Consider a function $f(\p)$, with $\p\in\mathbb{R}^2$, and assume a gradient Lipschitz constant $L=2$. Let the sample points be $\p_1 = \tran{[0,-0.5]}$ and $\p_2 = \tran{[0,0.5]}$, and set the gradient-error bound to $E^U = 2$. 
	
	Using the delta bound, the constraint $T_d(\p) \leq 2$ defines the two shaded feasible regions shown in the top plot of Figure~\ref{fig:TruncBounds}, located on either side of the hyperplane $\mathcal{H}$. The feasible regions corresponding to the radial bound, defined by the constraint $T_r(\p)\leq 2$, are shown in the middle plot of Figure~\ref{fig:TruncBounds}. As illustrated, the delta bound produces much smaller regions, highlighting its greater conservatism compared to the radial bound.
	
	Finally, the bottom plot of Figure~\ref{fig:TruncBounds} shows the feasible regions defined by the constraint $T_s(\p)\leq 2$, corresponding to the simplex bound. As expected, these regions are significantly larger than those defined by the radial bound.
	
	For all three bounds, the contour lines corresponding to a relaxed error limit $E^U=3$ are also included in Figure~\ref{fig:TruncBounds} to facilitate visual comparison of their scaling behavior.
\end{Example}

\begin{figure}
	
	\centering \includegraphics[width=6.5cm]{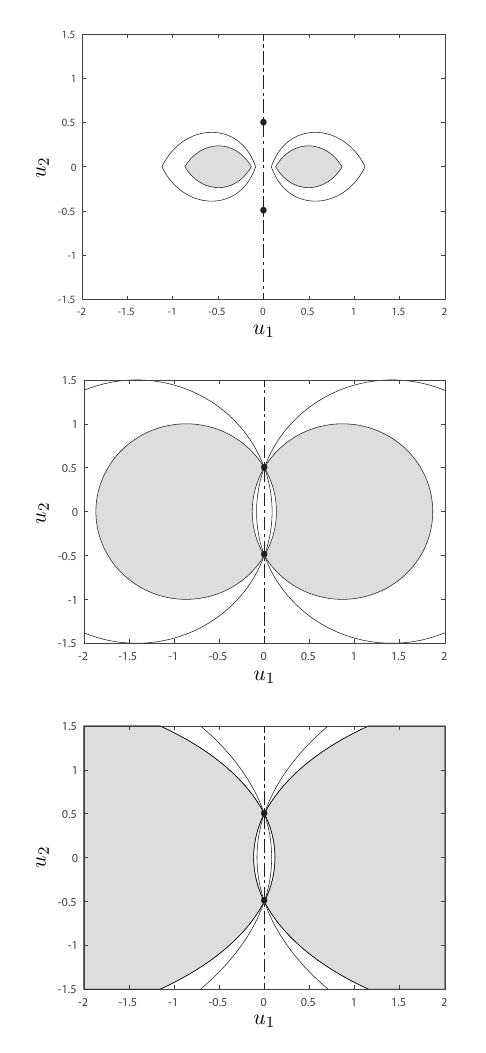} 
	
	\caption{Gradient error due to truncation with $L=2$, using the points $\p_1 = \tran{[0,-0.5]}$, $\p_2 = \tran{[0,0.5]}$. \underline{Top plot}: delta bound: feasible regions given by $T_d(\p)\leq 2$. \underline{Middle plot}: radial bound: feasible regions given by $T_r(\p)\leq 2$. \underline{Bottom plot}: simplex bound: feasible regions given by $T_s(\p)\leq 2$. The contour lines for $E^U=3$ are also shown in the three plots.} 
	\label{fig:TruncBounds}
\end{figure} 

\begin{Example}[Gradient error due to measurement noise] \label{ex:noiseerror}
	Let $\tilde{f}(\p) = f(\p) + \ru(\p)$ be a noisy function, where the noise is bounded as $|\ru(\p)|\leq \delta$. Let $\delta=0.2$, and consider a gradient-error bound of $E^U=2$. 
	
	The left-hand plots in Figure~\ref{fig:Noisebounds} correspond to the sample points $\p_1 = \tran{[0,-0.5]}$ and $\p_2 = \tran{[0,0.5]}$, while the right-hand plots use points that are closer together: $\p_1 = \tran{[0,-0.22]}$ and $\p_2 = \tran{[0,0.22]}$. 
	
	The top plots in Figure~\ref{fig:Noisebounds} show the feasible regions resulting from the conditioning-bound constraint $N_c(\p)\leq 2$, while the bottom plots correspond to the l-min bound constraint $N_l(\p)\leq 2$.
	
	As seen in the figure, the conditioning bound is more conservative than the l-min bound, and its conservatism increases markedly as the sample points become closer together. Contour lines corresponding to the relaxed threshold $E^U = 3$ are also included in all plots to facilitate comparison.
\end{Example}

\begin{figure} 
	
	\centering 
	
	\includegraphics[width=12cm]{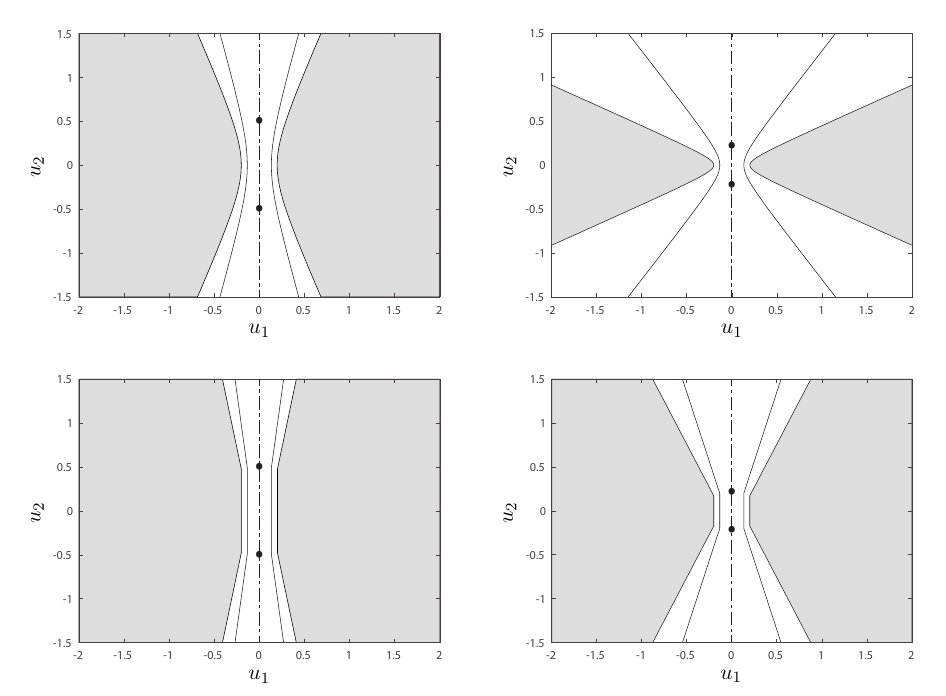} 
	
	\caption{Gradient error due to measurement noise with $\delta=0.2$. \underline{Left plots} use $\p_1 = \tran{[0,-0.5]}$ and $\p_2 = \tran{[0,0.5]}$. \underline{Right plots} use $\p_1 = \tran{[0,-0.22]}$ and $\p_2 = \tran{[0,0.22]}$. \underline{Top plots}: conditioning bound: feasible regions given by $N_c(\p)\leq 2$. \underline{Bottom plots}: l-min bound: feasible regions given by $N_l(\p)\leq 2$. The contour lines for $E^U=3$ are also shown in the plots.} 
	\label{fig:Noisebounds}
\end{figure} 

\begin{Example}[Total gradient error]
	Let $\tilde{f}(\p)$ be a noisy function with gradient Lipschitz constant $L=2$ and noise bounded by $\delta=0.2$. 
	
	The left-hand plots in Figure~\ref{fig:Totbounds} correspond to the sample points $\p_1 = \tran{[0,-0.5]}$ and $\p_2 = \tran{[0,0.5]}$, while the right-hand plots use more closely spaced points: $\p_1 = \tran{[0,-0.22]}$ and $\p_2 = \tran{[0,0.22]}$. 
	
	The top plots show the feasible regions defined by the constraint $E_r(\p) = T_r(\p)+N_l(\p)\leq 2$, which combines the radial bound with the l-min bound. The bottom plots show the regions resulting from $E_s(\p) = T_s(\p)+N_l(\p)\leq 2$, which combines the simplex bound with the l-min bound. As expected, the feasible regions associated with the simplex bound are considerably larger, reflecting its less conservative nature.
	
	Contour lines corresponding to a relaxed error bound $E^U = 3$ are also shown in all plots to facilitate comparison.
\end{Example}

\begin{figure}[tb]
	
	\centering 
	
	\includegraphics[width=12cm]{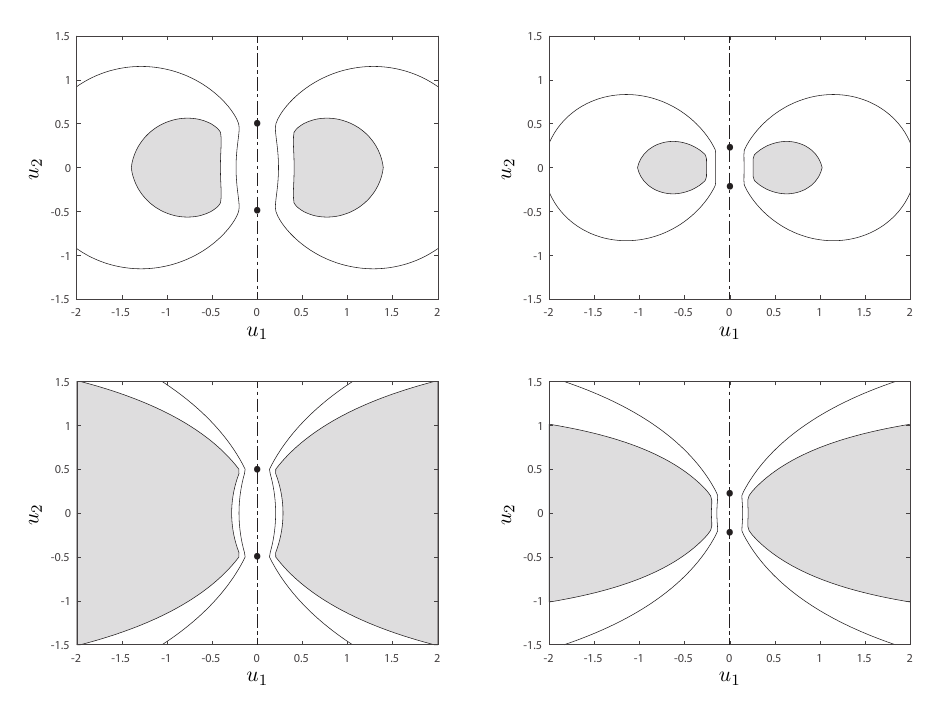} 
	
	\caption{Total gradient error with $L=2$ and $\delta=0.2$. \underline{Left plots} use $\p_1 = \tran{[0,-0.5]}$ and $\p_2 = \tran{[0,0.5]}$. \underline{Right plots} use $\p_1 = \tran{[0,-0.22]}$ and $\p_2 = \tran{[0,0.22]}$. \underline{Top plots}: feasible regions given by $E_r(\p)\leq 2$. \underline{Bottom plots}: feasible regions given by $E_s(\p)\leq 2$. The contour lines for $E^U=3$ are also shown in the plots.}  
	\label{fig:Totbounds}
\end{figure} 

The different gradient error bounds that have been presented and compared in this work could be useful in a variety of DFO schemes. In the next section, we present a sequential programming DFO scheme that directly includes the approximate gradient error bounds as constraints in the optimization problem.

%%%%%%%%%%%%%%%%%%%%%%%%%%%%%%%%%%%%%%%%%%%%%%%%%%%%%%%%%%%%%%%%%%%%%%%%%%%%%
\section{Derivative-Free Optimization of  Noisy Functions Using Duality Constraints}
\label{sec:DFOduality}

In this section, we propose a derivative-free optimization algorithm for noisy unconstrained optimization. We consider a sequential programming scheme that computes a step by minimizing the model. At each iteration, the model is defined by interpolating previously evaluated (noisy) function values. The key issue is to obtain reliable gradient estimates while progressing with the optimization. Emulating dual RTO algorithms \cite{Marchetti:2010,Marchetti:2013a}, gradient accuracy is enforced by including a constraint on the approximate gradient-error norm (duality constraint) in the optimization problem. Our proposed algorithm is rudimentary, lacking a claim to global convergence results.  Nevertheless, the main purpose here is to demonstrate the usefulness of duality constraints in DFO of noisy functions.\\

At each iteration $k$, our method uses the interpolation set $\mathcal{U}_k = \{\p_k,\p_{k-1},\dots,\p_{k-n_u}\}$. Based on this interpolation set, the model proposed at the $k^{th}$ iteration is
\begin{align}
	& m_k(\p) = q_k(\p) + \tilde{f}(\p_k) + \tran{\vect{\lambda}_k}(\p - \p_k), \label{eq:mk}
\end{align}
with
\begin{align}
	q_k(\p) =&\ \frac{1}{2}\tran{(\p-\p_k)}\vect{H}_k(\p-\p_k), \label{eq:qk} \\
	\tran{\vect{\lambda}_k} =&\ \tran{(\vect{y}_k - \vect{y}_{q,k})}(\vect{U}_k)^{-1}, \label{eq:lambdak} \\
	\vect{U}_k =& \left[\p_{k-1}-\p_k,\ \dots,\ \p_{k-n_u}-\p_k\right], \label{eq:Uk}\\
	\tran{\vect{y}_k} =& \left[\tilde{f}(\p_{k-1})-\tilde{f}(\p_{k}),\ \dots,\ \tilde{f}(\p_{k-n_u})-\tilde{f}(\p_k)\right], \label{eq:yk} \\
	\tran{\vect{y}_{q,k}} =& \left[q_k(\p_{k-1}),\ \dots,\ q_k(\p_{k-n_u})\right], \label{eq:yqk} 
\end{align}
where we used $q_k(\p_k) = 0$ in \eqref{eq:yqk}.
By construction, $m_k(\p)$ is the quadratic function with the Hessian matrix $\vect{H}_k$ that matches the noisy function evaluations $\tilde{f}(\p_j)$ at the points in $\mathcal{U}_k$.
Note that the simplex gradient at $\p_k$ is computed as
\begin{align}
	& \tran{\vect{g}_k} = \tran{\vect{y}_k}(\vect{U}_k)^{-1}. \nonumber
\end{align}

$\vect{H}_k$ is a positive definite approximation to the Hessian of $f$, which may be estimated from the previous function evaluations. A natural choice would be to implement the complementary BFGS update based on the gradient estimates $\vect{g}_k$ and $\vect{g}_{k-1}$ (see \cite{Bazaraa:06}). However, this may lead to erratic behavior and lack of convergence without further precautions.
In derivative-based sequential convex programming methods \cite{Beck:2010,Marks:78,Svanberg:2002,Marchetti:2017}, monotonic convergence to a local minimum of a noise-free function $f(\p)$ is established if, at each iteration $k$, the convex model $m_k(\p)$ satisfies: $i$) $\nabla m_k(\p_k) = \nabla f(\p_k)$, and $ii$) $m_k(\p)\geq f(\p)$ for all $\p\in\mathbb{R}^{n_u}$. This last condition can be achieved if, in any direction $\vect{d}\in\mathbb{R}^{n_u}$, the second derivative of $m_k$ upper bounds the second derivative of $f$. This convergence result does not extend to noisy functions with bounded gradient error. Nevertheless, a conservative Hessian will help enforce convergence as it leads to more conservative steps. We therefore select $\vect{H}_k = \vect{D} = L\vect{I}_{n_u}$ in our demonstration of the use of duality constraints in DFO.\\

The dividing hyperplane $\mathcal{H}_k$ and the duality constraint $E_k(\p)\leq E_k^U$, at the $k^{th}$ iteration, are computed from the $n_u$ previous points $\{\p_k,\p_{k-1},\dots,\p_{k-n_u+1}\}$. Let us introduce the matrix
\begin{align}
	& \vect{U}_{k,s} = \left[\p_{k-1}-\p_k,\ \p_{k-2}-\p_k,\dots,\p_{k-n_u+1}-\p_k\right] \in \mathbb{R}^{n_u\times(n_u-1)}. \nonumber
\end{align}

The normal vector $\vect{n}_k$ satisfies $\tran{(\vect{U}_{s,k})}\vect{n}_k=\vect{0}$, and the hyperplane is $\mathcal{H}_k = \{\p\in\mathbb{R}^{n_u}:\ \tran{\vect{n}_k}\p=b_k,\ \text{\rm with}\ b_k=\tran{\vect{n}_k}\p_k\}$. 

At the $k^{th}$ iteration, the approximate upper bound on the gradient error norm, denoted as $E_k(\p)$, is determined by the set of points $\{\p,\p_k,\p_{k-1},\dots,\p_{k-n_u+1}\}$. The duality constraint imposes the condition $E_k(\p)\leq E_k^U$.
This constraint controls gradient accuracy, depending on the chosen bounds for both the truncation gradient error and the error introduced by measurement noise. \\

The next operating point is computed by minimizing the model $m_k(\p)$ subject to the duality constraint, at each side of the hyperplane $\mathcal{H}_k$.
The optimization problem for the half space $\tran{\vect{n}_k}\p \geq b_k$ is
\begin{align}
	\p_{k+1}^{+} = \text{\rm arg}\min_{\p} \ \ &\ m_k(\p) \label{eq:optim1}\\
	\text{\rm s.t.} \quad &\ E_k(\p)\leq E_k^U,\quad \tran{\vect{n}_k}\p \geq b_k, \nonumber
\end{align}
while, for $\tran{\vect{n}_k}\p \leq b_k$, the problem reads
\begin{align}
	\p_{k+1}^{-} = \text{\rm arg}\min_{\p} \ \ &\ m_k(\p) \label{eq:optim2}\\
	\text{\rm s.t.} \quad &\ E_k(\p)\leq E_k^U,\quad \tran{\vect{n}_k}\p \leq b_k. \nonumber
\end{align}

The next point $\p_{k+1}$ is selected as the point that minimizes $m_k$,
\begin{align}
	\p_{k+1} = \text{\rm arg}\min_{\p_{k+1}^{+},\p_{k+1}^{-}}\ \{m_k(\p_{k+1}^{+}),m_k(\p_{k+1}^{-})\} \label{eq:minimizer}
\end{align}

The proposed DFO algorithm for noisy unconstrained optimization is formalized in Algorithm~\ref{alg:algor1}. 
\begin{algorithm}[h]
	\caption{Sequential Programming DFO with Duality Constraint \label{alg:algor1}}
	\begin{description}
		\item[Initialization:]
		Given a noisy function $\tilde{f}(\p)$, $\p\in\mathbb{R}^{n_u}$, with the gradient Lipschitz constant $L$, and noise bounded by $\delta$; Select the initial iterate $\p_0\in\mathbb{R}^{n_u}$. Set the iteration counter and the function evaluation counter to zero,   $k\leftarrow 0$ and $k_f\leftarrow 0$, respectively. Select a gradient-error bounding function $E(\p)$. Select the initial Hessian approximation $\vect{H}_0$ and the step size  $h$.
		\item {\bf 1.} Apply FFD at $\p_0$ by perturbing $\p_0$ in each input direction using the step size $h$. This results in the initial interpolation set $\mathcal{U}_0 = \{\p_0,\p_{-1},\dots,\p_{-{n_u}}\}$.
		\item {\bf 2.} Evaluate $\tilde{f}(\p)$ at $\p_0,\p_{-1},\dots,\p_{-{n_u}}$. Set $k_f = n_u+1$. 
		\item {\bf while} a convergence test is not satisfied {\bf do} 
		
		\item {\bf 3.} $\quad$ Define the interpolation set as $\mathcal{U}_k = \{\p_k,\p_{k-1},\dots,\p_{k-n_u}\}$.
		\item {\bf 4.} $\quad$ Define the model $m_k(\p)$ using \eqref{eq:mk}-\eqref{eq:yqk}.
		\item {\bf 5.} $\quad$ Determine the dividing hyperplane $\mathcal{H}_k$.
		\item {\bf 6.} $\quad$ Select the upper bound on the gradient-error norm $E_k^U$. 
		\item {\bf 7.} $\quad$ Solve Problems~\eqref{eq:optim1} and \eqref{eq:optim2}, and compute the next input $\p_{k+1}$ from \eqref{eq:minimizer}.
		\item {\bf 8.} $\quad$ Evaluate $\tilde{f}(\p_{k+1})$. Set $k_f=k_f+1$.
		\item {\bf 9.} $\quad$ Determine a new Hessian approximation $\vect{H}_{k+1}$.
		\item {\bf 10.} $\quad$ $k=k+1$.
		\item {\bf end while}
	\end{description}
\end{algorithm} 

In Examples~\ref{ex:casestudy1} and \ref{ex:casestudy2} below, we will consider two variants of Algorithm~\ref{alg:algor1} depending on the choice of the duality constraint. In selecting the upper bound $E_k^U$, it is very convenient to consider the minimum error bound for FFD using the radial bound and the l-min bound, which is given by
\begin{align}
	& E_{FFD}^\star := \frac{L\sqrt{n_u}\ h^\star}{2} + \frac{2\delta\sqrt{n_u}}{h^\star}, \quad\text{with}\ h^\star = 2\sqrt{\delta/L}. \label{eq:EFFSstar}
\end{align}

The two variants are defined by the following choices:\\

\noindent {\bf Algorithm 1a}
\begin{itemize}
	\item $E_k(\p) = E_{r,k}(\p) = T_{r,k}(\p) + N_{l,k}(\p)$.
	\item $E_k^U = \max\left\{\frac{\|\vect{g}_k\|}{4},\ E_{FFD}^\star\right\}$.
	\item $\vect{H}_k = L\vect{I}_{n_u}$.
\end{itemize}

\noindent {\bf Algorithm 1b}
\begin{itemize}
	\item $E_k(\p) = E_{s,k}(\p) = T_{s,k}(\p) + N_{l,k}(\p)$.
	\item $E_k^U = E_{FFD}^\star$.
	\item $\vect{H}_k = L\vect{I}_{n_u}$.
\end{itemize}

Algorithm 1a uses the radial bound. The gradient-error bound $E_k^U$ is selected as a fraction of the norm of the simplex gradient $\vect{g}_k$, and takes $E_{FFD}^\star$ as a minimum value. Algorithm 1b uses the simplex bound, and $E_k^U$ is kept constant and equal to $E_{FFD}^\star$. Both algorithms use the l-min bound to deal with measurement noise. As discussed earlier, the conservative Hessian $\vect{H}_k = L\vect{I}_{n_u}$ is adopted.

\begin{Example} \label{ex:casestudy1}
	Consider the function
	\begin{align}
		& f(\p) = 2\pk_1^2 - \pk_1\pk_2 + \pk_2^2 - 3\pk_1 + 1.4^{(2\pk_1+\pk_2)}, \nonumber
	\end{align}
	with $\p = \tran{[\pk_1,\ \pk_2]}$. A gradient Lipschitz constant that is valid in $\mathcal{Q} = \{\p\in\mathbb{R}^{2}:\ -2\leq \pk_1\leq 1,\ -2.5\leq\pk_2\leq 1\}$ is $L = 5.3$. We assume that the evaluations of $f$ are noisy, $\tilde{f}(\p) = f(\p)+\nu$, and the noise $\nu$ has a Gaussian distribution with standard deviation $\sigma_f=0.1$. The noise interval is selected as $\delta = 3\sigma_f=0.3$. Let us consider the two initial points $\p_0^a = \tran{[-2,\ -2.5]}$ and $\p_0^b = \tran{[-2,\ 0.5]}$. The first 60 iterates obtained using Algorithm~1a, which employs the radial bound, are shown in the left plot of Figure~\ref{fig:Ejemplo_2inputs}. 
	Both initial points lead to convergence to a neighborhood of the optimal solution. Similarly, the right plot in Figure~\ref{fig:Ejemplo_2inputs} displays the first 60 iterates of Algorithm~1b, which uses the simplex bound. 
	Due to this bound, Algorithm~1b allows for less restrictive movements and takes larger initial steps toward the optimum. Specifically, starting from point $\p_0^a$, Algorithm~1b reaches a neighborhood of the optimal solution in a single iteration, whereas Algorithm~1a requires five iterations. From point $\p_0^b$, both algorithms reach the neighborhood of the optimum in approximately nine iterations, but Algorithm~1b takes a larger initial step in that direction.
\end{Example}

\begin{figure}[h]
	
	\centering 
	
	\includegraphics[width=12cm]{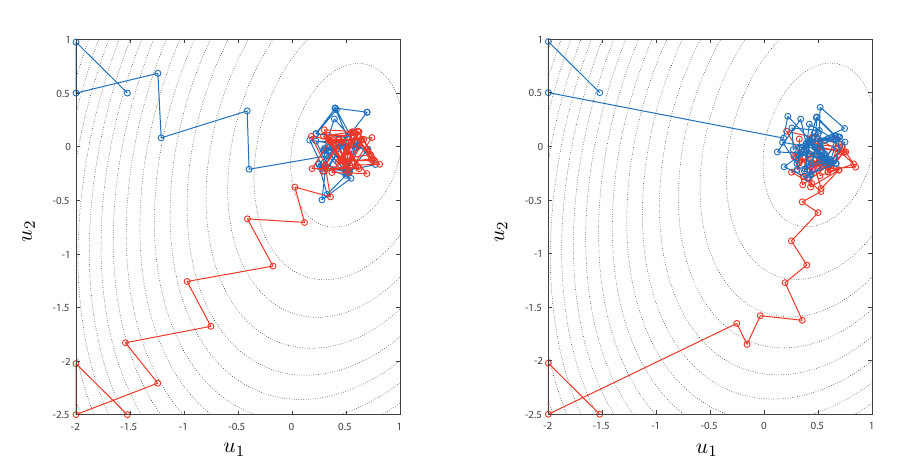} 
	
	\caption{Iterates in Example~\ref{ex:casestudy1} with two different initial points in red and blue. \underline{Left plot}: Algorithm~1a. \underline{Right plot}: Algorithm~1b.} 
	\label{fig:Ejemplo_2inputs}
\end{figure}

\begin{Example}  \label{ex:casestudy2}
	Consider the simple quadratic function
	\begin{align}
		& f(\p) = \frac{1}{2}\tran{\p}\vect{Q}\p, \nonumber
	\end{align}
	where $\p\in\mathbb{R}^3$, and $\vect{Q} = 2\vect{I}_3$. We assume that only noisy evaluations of $f$ are available, $\tilde{f}(\p) = f(\p)+\nu$, where the measurement noise $\nu$ has a Gaussian distribution with standard deviation $\sigma_f=0.05$. We select the noise interval $\delta = 3\sigma_f=0.15$, and the slightly conservative gradient Lipschitz constant $L=2.5$. Starting from the initial point $\p_0 = \tran{[2,\ 5,\ 3]}$, the first 60 iterates obtained using Algorithm~1a are shown in the top plot of Figure~\ref{fig:Ejemplo_3inputs}, while the first 60 iterates obtained using Algorithm~1b are shown in the bottom plot of Figure~\ref{fig:Ejemplo_3inputs}. Algorithm~1a converges to the neighborhood of the origin after 13 iterations (excluding the four initial FFD points), while Algorithm~1b reaches the neighborhood of the optimal solution already after 7 iterations.
\end{Example}

\begin{figure}[tbh]
	
	\centering 
	\includegraphics[width=15cm]{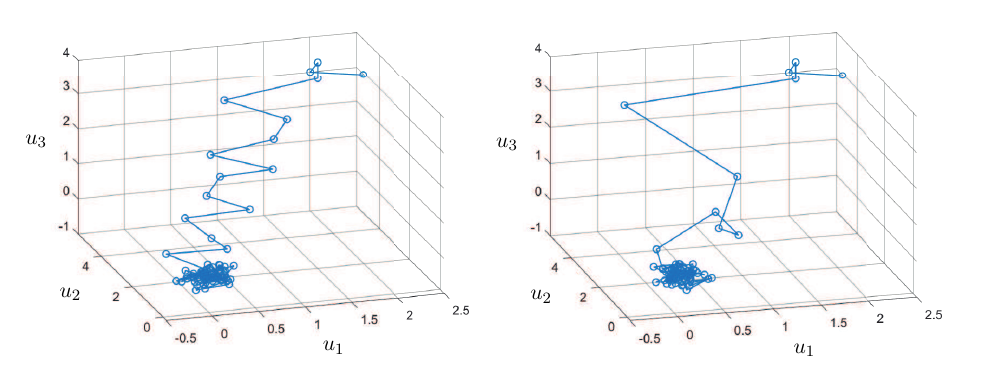} 
	
	\caption{Iterates in Example~\ref{ex:casestudy2}. \underline{Left plot}: Algorithm~1a. \underline{Right plot}: Algorithm~1b.} 
	\label{fig:Ejemplo_3inputs}
\end{figure} 

Note that the shape and size of the feasible regions generated by the duality constraint at each iteration depends on the position of the points $\{\p_k,\p_{k-1},\dots,\p_{k-n_u+1}\}$. The duality constraint ensures the accuracy of the simplex gradient evaluated at the next iteration.

The rationale behind defining the simplex bound becomes evident from the results.
Notably, equation~\eqref{eq:m_general} implies that the simplex gradient approximates the gradient at all points in the interpolation sample set $\mathcal{U}$, while Remark~\ref{rmk:radialuniform} indicates that the radial bound takes the same value at all points in $\mathcal{U}$. This means that as we approach a stationary point, every point in the interpolation sample set serves as an equally valid estimate of that stationary point. Consequently, convergence occurs with a final simplex rather than a final iterate.

By using the simplex bound, the focus shifts from ensuring gradient accuracy at the last iterate to maintaining accuracy at any point within the last simplex. The simplex bound is considerably less conservative than the radial bound, allowing the duality constraint to define larger feasible regions for placing the next point. As a result, Algorithm 1b permits less restrictive movements than Algorithm 1a, enabling both larger and smaller steps. This often leads to faster convergence to a neighborhood of the plant optimum, though not always. For example, in the left plot of Figure 10, starting from $(-2,\ -2.5)$, Algorithm 1b initially takes a large step but then switches to smaller steps than Algorithm 1a, likely due to encountering a region with a gentler slope.

%%%%%%%%%%%%%%%%%%%%%%%%%%%%%%%%%%%%%%%%%%%%%%%%%%%%%%%%%%%%%%%%%%%%%%%%%%%%%
\section{Conclusions}
\label{sec:Conclusions}

We have examined the accuracy of the simplex gradient obtained via linear interpolation of $n_u+1$ points for functions subject to uniformly bounded noise. Various upper bounds on the gradient error, arising from both truncation and measurement noise, have been derived and systematically compared.

An important contribution of this work is the square column bound, a novel upper bound on the norm of the gradient error due to truncation. This bound is shown to be less conservative than the commonly used delta bound, with both bounds coinciding when the columns of the matrix $\vect{U}$ are of equal length. As demonstrated in Example~\ref{ex:3}, the delta bound can become significantly more conservative than the square column bound when the domain dimension $n_u$ increases and the columns of $\vect{U}$ differ in norm.

We also provide new theoretical results for the radial bound, showing that it upper bounds the projection of the gradient error along the directions defined by the columns of $\vect{U}$, and becomes a strict upper bound on the full gradient error norm when these directions are orthogonal, even if not normalized. 

Examples~\ref{ex:1}–\ref{ex:3} and \ref{ex:Ex5} highlight the excessive conservatism of strict bounds like the delta and square column bounds in many cases. In contrast, while the radial bound does not guarantee a bound in all scenarios, it has desirable structural properties and frequently provides a tight and less conservative approximation of the gradient error. For these reasons, we classify the radial bound as an approximate gradient error bound. We also define the extended radial bound for non-vertex points within the simplex, and the simplex bound as the minimum value that the extended radial bound attains inside the simplex.

Both the square column and radial bounds are simple to compute and are directly applicable to assessing gradient accuracy in DFO schemes for noise-free functions. However, all bounds that rely on a gradient Lipschitz constant may become overly conservative—particularly for functions with highly variable Hessians—and, since this constant is often unknown in practice, effective data-driven estimation strategies remain essential.

For gradient errors due to measurement noise, Example~\ref{ex:noiseerror} confirms that the conditioning bound is notably more conservative than the l-min bound, especially when the sample points are closely spaced. While computing the l-min bound involves evaluating distances between complementary affine subspaces—a number that grows rapidly with $n_u$—the computational burden remains manageable in many real-world applications, where the number of independent adjustable variables is typically modest (2 to 10). Still, the l-min bound represents a worst-case scenario, which may be overly pessimistic in practice, particularly for higher-dimensional problems. Developing gradient-error bounds that are both computationally efficient and less conservative than the l-min bound is an important avenue for future research.

We also introduced a sequential programming DFO scheme that incorporates approximate gradient error bounds as duality constraints for unconstrained optimization of noisy functions. The results from Examples~\ref{ex:casestudy1} and \ref{ex:casestudy2} demonstrate the effectiveness of this strategy. By ensuring gradient accuracy at at least one point in the last simplex, the simplex bound allows for significantly less conservative duality constraints compared to the radial bound. This leads to larger permissible step sizes and faster convergence to a neighborhood of the plant optimum, without compromising gradient accuracy. 

Several promising directions for future research emerge from this study. The square column bound, radial bound, l-min bound, extended radial bound, and simplex bound remain largely unexplored in the DFO literature. In particular, the square column and radial bounds could offer valuable alternatives to the delta bound for certifying model quality in derivative-free trust-region methods. Further research could also extend these bounds to the underdetermined and overdetermined settings. Moreover, the concept of duality constraints appears especially promising for DFO under noisy conditions. Extending Algorithm~1 to handle constrained optimization problems and developing convergence guarantees toward a neighborhood of the optimal solution are key topics for future investigation.

%%%%%%%%%%%%%%%%%%%%%%%%%%%%%%%%%%%%%%%%%%%%%%%%%%%%%%%%%%%%%%%%%%%%%%%%%%%%%

\color{black}

\appendix
\renewcommand{\theequation}{\thesection.\arabic{equation}}

\section{Complement Affine Subspaces} \label{app:affine_subspaces}

In a $\np$-dimensional space, a point is an affine subspace of dimension 0, a line is an affine subspace of dimension 1, and a plane is an affine subspace of dimension 2. An affine subspace of dimension $(\np-1)$ is called an hyperplane.\\

{\it Hyperplane.}
An hyperplane in $\np$-dimensional space is given by
\begin{align}
	& n_1u_1 + n_2u_2 + \dots + n_{\np}u_{\np}=b,  \quad \text{or:} \quad \tran{\en}\p=b \label{equ:hyper}
\end{align}
and divides the space into two half-spaces: $\tran{\en}\p \geq b$, and $\tran{\en}\p \leq b$.\\

\begin{Definition}
	{\bf (Complement affine subspaces)}
	Given a set of $(\np+1)$ points in a $\np$-dimensional space, $\mathcal{U} := \{\p_0, \p_1, \dots,\p_{\np}\}$, a proper subset $\mathcal{U}_A$ of $\np^A \in \{1,\dots,n_u\}$ points determines an affine subspace of dimension $(\np^A-1)$.
	The complement subset $\mathcal{U}_C := \mathcal{U}\setminus\mathcal{U}^A$ of $(\np+1-\np^A)$ points, determines the complement affine subspace of dimension $(\np-\np^A)$.
\end{Definition}

\begin{Definition}
	{\bf (Distance between complement affine subspaces)}
	%(Distance between complement affine subspaces).
	Given a set of $(\np+1)$ points in a $\np$-dimensional space, $\mathcal{U} := \{\p_0, \p_1,\dots,\p_{\np}\}$, a proper subset of $\mathcal{U}$, $\mathcal{U}_A \subsetneq \mathcal{U}$ of $\np^A \in \{1,\dots,\np\}$ points, and its complement $\mathcal{U}_C := \mathcal{U}\setminus\mathcal{U}_A$ of $(\np+1-\np^A)$ points, the distance between complement affine subspaces is defined as the (orthogonal) distance between the affine subspace of dimension $(\np^A-1)$ generated by all the points in $\mathcal{U}_A$, and the affine subspace of dimension $(\np-\np^A)$ generated by all the points in $\mathcal{U}_C$. 
\end{Definition}

The total number of possible pairs of complement affine subspaces that can be generated from $\mathcal{U}$ is 
\begin{align}
	& n_b=1+\sum_{s=1}^{\np-1}2^s. \nonumber
\end{align}

Table~\ref{tab:numb} gives the value of $n_b$ for values of $n_u$ from 1 to 10.
\begin{table}[h]
	\caption{Number of complement affine subspaces} %title of the table
	\centering % centering table
	\begin{tabular}{c cccccccccc} % creating eight columns
		\hline %inserting double-line
		$n_u$ & 1 & 2 & 3 & 4 & 5 & 6 & 7 & 8 & 9 & 10 \\ 
		$n_b$ & 1 & 3 & 7 & 15 & 31 & 63 & 127 & 255 & 511 & 1023\\ % Entering row contents
		\hline % inserts single-line
	\end{tabular}
	\label{tab:numb}
\end{table}

In the 2-dimensional case ($\np=2$), the 3 distances to evaluate correspond to the 3 point-to-line distances. In the 3-dimensional case, there are $n_b=7$ distances to evaluate, which correspond to 4 point-to-plane distances, and 3 line-to-line distances. \\

In order to compute the distance between the complement affine subspaces determined by the points in $\mathcal{U}_A$ and $\mathcal{U}_C$, a vector $\en$ that is normal to both subspaces is required:
\begin{align} 
	&\tran{[\ \p_1-\p_0,\ \dots\ , \ \p_{\np^A-1}-\p_0,\ \ \p_{\np^A+1}-\p_{\np^A},\ \dots\ , \p_{\np}-\p_{\np^A} \ ]} \en = \vect{0}, \label{equ:calcnormal}
\end{align}
or
\begin{align}
	& \vect{Q} \en = \vect{0}. \nonumber
\end{align}

	The matrix $\vect{Q} \in\mathbb{R}^{(\np-1)\times\np}$ is of rank $(\np-1)$. The vector $\en$ can be obtained by singular-value decomposition of $\vect{Q}$.
	
	Given a point $\p^a\in\mathcal{U}_A$, a point $\p^b\in\mathcal{U}_C$, and a vector $\en$ that is normal to both complement affine subspaces, the distance $l_{AC}$ between the complement affine subspaces is:
	\begin{align}
		l_{AC} = \frac{\lvert\tran{\en}(\p^b-\p^a)\rvert}{\| \en \|} \label{equ:calcdist}
	\end{align}

\begin{Definition}
	{\bf (Nearest Complement Affine Subspaces)}
	%[Nearest Complement Affine Subspaces]
	The shortest distance between complement affine subspaces is given by $l_{\text{\rm min}}:=\text{\rm min}\{l_1,l_2,\dots,l_{n_b}\}$, where $l_1,l_2,\dots,l_{n_b}$ are the distances between all possible pairs of complement affine subspaces that can be generated from $\mathcal{U}$.
\end{Definition}

%%%%%%%%%%%%%%%%%%%%%%%%%%%%%%%%%%%%%%%%%%%%%%%%%%%%%%%%%%%%%%%%%%%%%%%%%%%%%
%\clearpage

%\bibliographystyle{elsart-num}
\bibliographystyle{plainnat}

\bibliography{Marchetti_2025}  

\begin{thebibliography}{23}
\providecommand{\natexlab}[1]{#1}
\providecommand{\url}[1]{\texttt{#1}}
\expandafter\ifx\csname urlstyle\endcsname\relax
  \providecommand{\doi}[1]{doi: #1}\else
  \providecommand{\doi}{doi: \begingroup \urlstyle{rm}\Url}\fi

\bibitem[Bazaraa et~al.(2006)Bazaraa, Sherali, and Shetty]{Bazaraa:06}
M.~S. Bazaraa, H.~D. Sherali, and C.~M. Shetty.
\newblock \emph{Nonlinear Programming: Theory and Algorithms}.
\newblock John Wiley and Sons, New Jersey, 3rd edition, 2006.

\bibitem[Beck et~al.(2010)Beck, Ben-Tal, and Tetruashvili]{Beck:2010}
A.~Beck, A.~Ben-Tal, and L.~Tetruashvili.
\newblock A sequential parametric convex approximation method with applications
  to non convex truss topology design problems.
\newblock \emph{Journal of Global Optimization}, 47:\penalty0 29--51, 2010.

\bibitem[Berahas et~al.(2022)Berahas, Cao, Choromanski, and
  Scheinberg]{Berahas:2022}
A.~S. Berahas, L.~Cao, K.~Choromanski, and K.~Scheinberg.
\newblock A theoretical and empirical comparison of gradient approximations in
  derivative-free optimization.
\newblock \emph{Foundations of Computational Mathematics}, 22:\penalty0
  507--560, 2022.

\bibitem[Bortz and Kelley(1998)]{Bortz:1998}
D.~M. Bortz and C.~T. Kelley.
\newblock \emph{The Simplex Gradient and Noisy Optimization Problems}.
\newblock In: Borggaard, Jeff and Burns, John and Cliff, Eugene and Schreck,
  Scott. Computational Methods for Optimal Design and Control, Boston, MA,
  1998.

\bibitem[Brekelmans et~al.(2005)Brekelmans, Driessen, Hamers, and den
  Hertog]{Brekelmans:05}
R.~C.~M. Brekelmans, L.~T. Driessen, H.~L.~M. Hamers, and D.~den Hertog.
\newblock Gradient estimation schemes for noisy functions.
\newblock \emph{Journal of Optimization Theory and Applications}, 126\penalty0
  (3):\penalty0 529--551, 2005.

\bibitem[Conn et~al.(2009)Conn, Scheinberg, and Vicente]{Conn:2009}
A.~R. Conn, K.~Scheinberg, and L.~N. Vicente.
\newblock \emph{Introduction to Derivative-Free Optimization}.
\newblock MPS-SIAM Optimization Series, Philadelphia, USA, 2009.

\bibitem[Cust\'odio et~al.(2008)Cust\'odio, Jr, and Vicente]{Custodio:2008}
A.~L. Cust\'odio, J.~E.~Dennis Jr, and L.~N. Vicente.
\newblock Using simplex gradients of nonsmooth functions in direct search
  methods.
\newblock \emph{IMA Journal of Numerical Analysis}, 28:\penalty0 770--784,
  2008.

\bibitem[Hare et~al.(2020)Hare, Jarry-Bolduc, and Planiden]{Hare:2020}
W.~Hare, G.~Jarry-Bolduc, and C.~Planiden.
\newblock Error bounds for the overdetermined and underdetermined generalized
  centred simplex gradients.
\newblock \emph{IMA Journal of Numerical Analysis}, 41\penalty0 (1):\penalty0
  744--770, 2020.

\bibitem[Jarry-Bolduc(2023)]{Jarry:2023}
Gabriel Jarry-Bolduc.
\newblock \emph{Numerical Analysis for Derivative-Free Optimization}.
\newblock PhD thesis, University of British Columbia, 2023.

\bibitem[Kelley(1999)]{Kelley:1999}
C.~Kelley.
\newblock \emph{Iterative {M}ethods for {O}ptimization}, volume~18.
\newblock SIAM, 1999.

\bibitem[Marchetti et~al.(2009)Marchetti, Chachuat, and Bonvin]{Marchetti:09}
A.~Marchetti, B.~Chachuat, and D.~Bonvin.
\newblock Modifier-adaptation methodology for real-time optimization.
\newblock \emph{Industrial \& Engineering Chemistry Research}, 48\penalty0
  (13):\penalty0 6022--6033, 2009.

\bibitem[Marchetti et~al.(2010)Marchetti, Chachuat, and Bonvin]{Marchetti:2010}
A.~Marchetti, B.~Chachuat, and D.~Bonvin.
\newblock A dual modifier-adaptation approach for real-time optimization.
\newblock \emph{Journal of Process Control}, 20:\penalty0 1027--1037, 2010.

\bibitem[Marchetti(2013)]{Marchetti:2013a}
A.~G. Marchetti.
\newblock A new dual modifier-adaptation approach for iterative process
  optimization with inaccurate models.
\newblock \emph{Computers and Chemical Engineering}, 59:\penalty0 89--100,
  2013.

\bibitem[Marchetti(2022)]{Marchetti:2022}
A.~G. Marchetti.
\newblock Feasibility in real-time optimization under model uncertainty. {T}he
  use of {L}ipschitz bounds.
\newblock \emph{Computers and Chemical Engineering}, 168:\penalty0 108057,
  2022.

\bibitem[Marchetti et~al.(2016)Marchetti, Fran\c{c}ois, Faulwasser, and
  Bonvin]{Marchetti:2016}
A.~G. Marchetti, G.~Fran\c{c}ois, T.~Faulwasser, and D.~Bonvin.
\newblock Modifier adaptation for real-time optimization -- {M}ethods and
  applications.
\newblock \emph{Processes}, 4\penalty0 (55), 2016.

\bibitem[Marchetti et~al.(2017)Marchetti, Faulwasser, and
  Bonvin]{Marchetti:2017}
A.~G. Marchetti, T.~Faulwasser, and D.~Bonvin.
\newblock A feasible-side globally convergent modifier-adaptation scheme.
\newblock \emph{Journal of Process Control}, 54:\penalty0 38--46, 2017.

\bibitem[Marks and Wright(1978)]{Marks:78}
B.~R. Marks and G.~P. Wright.
\newblock A general inner approximation algorithm for nonconvex mathematical
  programs.
\newblock \emph{Operations Research}, 26\penalty0 (4):\penalty0 681--683, 1978.

\bibitem[Mor\'e and Wild(2012)]{More:2012}
J.~J. Mor\'e and S.~M. Wild.
\newblock Estimating derivatives of noisy simulations.
\newblock \emph{ACM Transactions on Mathematical Software (TOMS)}, 39\penalty0
  (3):\penalty0 1--21, 2012.

\bibitem[Nesterov(2004)]{Nesterov:2004}
Y.~Nesterov.
\newblock \emph{Introductory {L}ectures on {C}onvex {O}ptimization. {A} {B}asic
  {C}ourse}.
\newblock Kluwer Academic Publishers, United Kingdom, 2004.

\bibitem[Regis(2015)]{Regis:2015}
Rommel~G. Regis.
\newblock The calculus of simplex gradients.
\newblock \emph{Optimization Letters}, 9:\penalty0 845--865, 2015.

\bibitem[Rios and Sahinidis(2013)]{Rios:2013}
L.M. Rios and N.~V. Sahinidis.
\newblock Derivative-free optimization: A review of algorithms and comparison
  of software implementations.
\newblock \emph{Journal of Global Optimization}, 56:\penalty0 1247--1293, 2013.

\bibitem[Rockafellar(1997)]{Rockafellar:97}
R.~T. Rockafellar.
\newblock \emph{Convex Analysis}.
\newblock Princeton University Press, Princeton, New Jersey, 1997.

\bibitem[Svanberg(2002)]{Svanberg:2002}
K.~Svanberg.
\newblock A class of globally convergent optimization methods based on
  conservative convex separable approximations.
\newblock \emph{SIAM J. Optim.}, 12\penalty0 (2):\penalty0 555--573, 2002.

\end{thebibliography}

\end{document}